\documentclass[11pt]{article}
\usepackage[bib=library]{mathesis}

\date{\today}
\author{
{\large Fabio Nobile\textsuperscript{a}, Raúl Tempone\textsuperscript{b}, Sören Wolfers\textsuperscript{b}\footnote{Corresponding author. Email address: \texttt{soeren.wolfers@kaust.edu.sa}}}\\
{\small \textsuperscript{a} École Polytechnique Fédérale de Lausanne (EPFL), CSQI-MATHICSE}\\
{\small \textsuperscript{b} King Abdullah University of Science and Technology (KAUST), CEMSE}\\
}

\title{Sparse approximation of multilinear problems with applications to kernel-based methods in UQ}

\newcommand{\vsQOI}{\R}

\newcommand{\op}{\Xi}

\newcommand{\funct}{\qoi}
\newcommand{\fq}{q}

\newcommand{\dom}{\Gamma}
\newcommand{\supdom}{\Lambda}

\newcommand{\Riesz}{\mathcal{R}}
\newcommand{\native}[2]{\mathcal{N}_{#1}(#2)}
\newcommand{\mix}{\text{mix}}
\newcommand{\Fourier}[1]{\hat{#1}}
\newcommand{\reg}{\beta}
\newcommand{\regs}{\alpha}

\newcommand{\vsPS}{H}

\newcommand{\domPS}{\Gamma}
\newcommand{\factor}{W}
\newcommand{\work}{\text{Work}}
\newcommand{\smoll}{\mathcal{S}}
\newcommand{\smol}{\smoll_L(\goal)}
\newcommand{\qoi}{Q}
\newcommand{\goal}{v}
\newcommand{\inp}{w}
\newcommand{\vsgoal}{V}

\newcommand{\bl}{\mathcal{M}}
\newcommand{\bbl}{\mathcal{B}}

\newcommand{\pa}{D}
\newcommand{\dt}{n}
\newcommand{\ns}{N}

\newcommand{\emp}{\Xi}
\newcommand{\qr}{\mathcal{I}}
\newcommand{\ind}{\mathcal{J}}

\newcommand{\PDE}{P}
\newcommand{\app}{\inp}

\usepackage{pdfpages}
\newcommand{\regnum}{2}

\begin{document}
\maketitle
\abstract{

	We provide a framework for the sparse approximation of multilinear problems and show that several problems in uncertainty quantification fit within this framework. In these problems, the value of a multilinear map has to be approximated using approximations of different accuracy and computational work of the arguments of this map. We propose and analyze a generalized version of Smolyak’s algorithm, which provides sparse approximation formulas with convergence rates that mitigate the curse of dimension that appears in multilinear approximation problems with a large number of arguments. We apply the general framework to response surface approximation and optimization under uncertainty for parametric partial differential equations using kernel-based approximation. The theoretical results are supplemented by numerical experiments.\\[1em]
 
 

}
\pagenumbering{roman}
\pagestyle{headings}
\pagenumbering{arabic}

\section{Introduction}
In the first part of this work, we consider the problem of approximating the value
\begin{equation*}
\goal:=\bl(\inp^{(1)},\dots,\inp^{(\dt)})
\end{equation*}
of a continuous multilinear map $\bl$, given approximations $\app^{(j)}_{\ns}$, $N\geq 0$ of the inputs $\inp^{(j)}$ for $j\in\{1,\dots,n\}$. We assume that as $N$ grows, the accuracy of the approximations $w^{(j)}_{N}$ increases but that simultaneously the required work goes to infinity.

In practice, the map $\bl$ may be as simple as the application of a linear operator to a real-valued function on a domain $\domPS\subset\R^d\to \R$. For example, approximating the identity operator $w^{(1)}:=\Id$ by an interpolation operator based on evaluations in $X\subset\Gamma$ with $|X|=N$, and approximating a function $w^{(2)}:=f$ by $f_M\colon \domPS\to\R$ with $f_M\stackrel{M\to\infty}{\longrightarrow} f$, yields an approximation of $v:=f=\Id f=:\bl(\Id,f)$ that is based on $N$ samples of $f_M$. For this approximation to be accurate we need both a large number $N$ of samples and a large $M$ such that $f_M$ is close to $f$.

In general, a straightforward approach to estimate $v$ is to consider 
\begin{equation*}
\goal_{\ns}:=\bl(\app^{(1)}_{\ns},\dots,\app^{(\dt)}_{\ns}),
\end{equation*}
and let $\ns\to\infty$.
If 
$$
\|\inp^{(j)}-\app^{(j)}_{\ns}\|\leq \ns^{-\beta}
$$ for all $j\in\{1,\dots,\dt\}$, then an induction argument shows 
\begin{equation*}
\|\goal-\goal_{\ns}\|\leq C \ns^{-\beta}.
\end{equation*}

However, if the work required to evaluate $\app^{(j)}_{\ns}$, $j\in\{1,\dots,\dt\}$, grows like $\ns^{\gamma}$, and if the cost for evaluating the multilinear map is multiplicative, then the work required to form $\goal_{\ns}$ is
\begin{equation*}
\work(\goal_{\ns})\approx \ns^{\dt\gamma}.
\end{equation*}

This is an instance of the \emph{curse of dimensionality}: To achieve the same error in $n$ dimensions as in $1$ dimension, the work needs to be exponentiated. 

Smolyak's algorithm was introduced in \cite{smolyak1963quadrature} and further studied in \cite{Wahba1978,Wasilkowski1995,NovakRitter1996,GerstnerGriebel1998} for the case where the multilinear map is given by the tensor product of quadrature and interpolation formulas. It allows an error of size $\exp(-\beta L)L^{\dt-1}$ with work $\exp({\gamma}L)L^{\dt-1}$ \cite[Lemma 2, Lemma 7]{Wasilkowski1995}. This means that the work required to achieve an error of size $\epsilon>0$ is bounded by $\epsilon^{-\gamma/\beta}|\log\epsilon|^{(n-1)(1+\gamma/\beta)}$. Therefore, up to logarithmic factors and multiplicative dimension-dependent constants, the curse of dimensionality has been lifted. More recently, it was shown \cite{griebel2013note,GriebelHarbrecht2013,dung2000continuous} that if the rates $\beta_j$, $\gamma_j$ differ with $j\in\{1,\dots,\dt\}$ the work can further be reduced to $\epsilon^{-\rho}$ if only one input has the maximal exponent ratio $\rho=\max_{j=1}^{\dt}\gamma_{j}/\beta_{j}$. Our results generalize this analysis to the case of general multilinear approximation problems. For example, while the work \cite{griebel2013note} exploits multiscale hierarchies and orthogonal decompositions to construct sparse wavelet approximation spaces, our results apply to arbitrary approximation schemes and thus provide conceptual simplifications that are helpful for both theoretical analysis and general purpose numerical implementations. Furthermore, our results are not restricted to quadrature and interpolation problems, but apply to rather general numerical approximation problems with multiple discretization parameters.

When applied to the example described at the beginning of this introduction, namely $v=\bl(\Id,f)$, Smolyak's algorithm yields a multilevel algorithm that combines samples from different approximations of $f$ with the general idea that more samples are taken from less expensive approximations, the exact numbers being determined by the involved work and convergence rates.  A connection between multilevel methods and Smolyak's algorithm has been discussed previously in \cite{harbrecht2012multilevel,harbrecht2013multilevel}; however, the results there were formulated only for quadrature problems and the analysis was based solely on balancing errors of the involved approximations, ignoring the associated computational work.
 \newline

In the second part of this work, we show how several problems in uncertainty quantification can be cast as multilinear approximation problems and tackled using the general Smolyak algorithm. We demonstrate how multilevel \cite{giles2008multilevel,TeckentrupJantschWebsterEtAl2015,kuo2017multilevel,Giles2015} and multi-index \cite{haji2015multi} methods for the approximation of expectations can be regarded as instances of Smolyak's algorithm, and we obtain novel methods for kernel-based response surface approximation and optimization under uncertainty with improved theoretically guaranteed convergence rates when compared to straightforward approaches.

We study parametric partial differential equations, i.e., problems of the form
\begin{align}
\label{eq:PDEIntro}
\PDE_{y}(u_{y})=f_{y},
\end{align}
where both the partial differential operator $\PDE_{y}$ and the right-hand side $f_{y}$ depend on a parameter $y$, and we are interested in a possibly nonlinear real-valued quantity of interest $\qoi(u_{y})$. For example, \Cref{eq:PDEIntro} may model physical problems with parameters describing material properties, boundary conditions and forcing terms, and $\qoi$ may be a spatial average or a point value.

For numerical computations, at least two types of approximation are required. For each given parameter $y$, we can only compute solutions $u_{y,N}$ of $u_{y}$ stemming from discretizations of \Cref{eq:PDEIntro}, with $u_{y,N}\to u_y$ as $N\to\infty$. Furthermore, we can only compute such approximations for finitely many values of the parameter $y$. The straightforward approach to obtain small errors is to take a sufficiently large number of samples computed with a sufficiently fine discretization of \Cref{eq:PDEIntro} and then use a suitable interpolation method to obtain estimates of the quantity of interest for intermediate parameter values. To improve on this approach, we treat each approximation as one factor within a multilinear approximation problem. Smolyak's algorithm then yields multilevel methods that combine many samples of coarse approximations of the PDE with fewer samples of finer approximations. Such methods were studied for the computation of expected values using different quadrature methods in \cite{heinrich2001multilevel,giles2008multilevel,harbrecht2012multilevel,TeckentrupJantschWebsterEtAl2015,kuo2017multilevel}. If the parameter space is a finite-dimensional product domain, then one may include this structure into the multilinear approximation problem and Smolyak's algorithm coincides with the  \emph{Multi-index Stochastic Collocation Method} of \cite{Haji-AliNobileTamelliniEtAl2015a,Haji-AliNobileTamelliniEtAl2015}. We extend these methods to the approximation of the full response surface, which allows among others for the computation of higher statistical moments, for application to inverse problems and for the optimization of parameters. Using the general theory for Smolyak's algorithm, we obtain convergence rates that essentially only reflect the constituent approximation with the worst complexity. For example, when the response surface is smooth enough, then approximations of the full response surface can be obtained at the same cost as response approximations for one single value of the parameter. To construct approximations of the response surface, we use kernel-based approximation \cite{SchabackWendland2006}, for which we provide the required background in \Cref{sec:kernel}. As a by-product, we obtain novel bounds also for single-level kernel-based approximation on sparse grids. Previous work in this direction \cite{Schreiber2000} established convergence bounds in the $L^{\infty}$-norm; we extend these bounds to Sobolev norms and more general function spaces. We note in passing that the choice of kernel-based approximations is not crucial. Indeed, any interpolation or approximation method that provides operators converging to the identity in some appropriate operator norm may be used. The strengths of kernel-based approximation are that the domain is not restricted to be an interval or hypercube, the data can be given in an unstructured form (i.e. not on a grid), and more general types of information (e.g. derivative values) can easily be included to enhance the resulting approximation.

Finally, we consider problems where the parameter $y=(z,m)$ of \Cref{eq:PDEIntro} can be split into a deterministic part $z\in\domPS\subset\R^d$ and a random part $m$. For example, this situation is studied in optimization under uncertainty \cite{Sahinidis2004,Shapiro2008}, which is concerned with problems of the form
\begin{equation}
\label{eq:minIntro}
\min_{z}E[Q(u_{(z,m)})]+\psi(z),
\end{equation}
where $Q$ is a quantity of interest that is sought to be minimized and $\psi(z)$ represents costs that are associated with the control $z$.

We show how response surface approximation for $z$ (using kernel-based approximation), expectations over $m$ (using Monte Carlo sampling) and numerical approximation of the PDE (using black-box PDE solvers) can be treated jointly in a multilinear approximation setting. Applying Smolyak's algorithm then gives rise to a novel method for the approximative solution of \eqref{eq:minIntro}. More specifically, we obtain surrogate models that can be evaluated at low cost such that standard minimization procedures can be applied. Under some assumptions, these surrogate models converge to the true model at the rate of Monte Carlo methods, which means that the work required for approximation of the PDE and for interpolation between finitely many choices of the parameter becomes negligible.

The remainder of this work is organized as follows. In \Cref{sec:sparse}, we introduce multilinear approximation problems and analyze Smolyak's algorithm applied to this setting. In \Cref{sec:kernel} we provide a short introduction to kernel-based approximation and show how mixed regularity gives rise to a multilinear structure that may be exploited using the results of \Cref{sec:sparse}, yielding kernel-based approximation with sparse grids. In \Cref{sec:UQ}, we study the numerical approximation of parametric and random PDEs. Applying the results from \Cref{sec:sparse}, we obtain novel sparse kernel-based approximation methods with theoretically guaranteed convergence rates. In \Cref{sec:numerics}, we present numerical experiments on parametric, linear elliptic PDE problems that confirm the theoretical convergence rates for response surface approximation and optimization under uncertainty.
\section{Sparse approximation of multilinear problems}
\label{sec:sparse}

Suppose we want to approximate the value
\begin{equation*}
\goal:=\bl( \inp^{(1)},\dots,\inp^{(\dt)})
\end{equation*} of a multilinear map
\begin{equation*}
\bl\colon \factor^{(1)}\times\dots \times \factor^{(\dt)}\to \vsgoal,
\end{equation*}
where $\factor^{(1)},\dots,\factor^{(\dt)}$ and $\vsgoal$ are normed spaces, and $\inp^{(j)}$ are fixed but not available inputs for which we are given approximations $\app^{(j)}_{\ns}\stackrel{N\to\infty}{\longrightarrow} \inp^{(j)}$,  $j\in\{1,\dots,\dt\}$.

In the applications that we consider in this work (see \Cref{sec:kernel,sec:UQ}), each input $\inp^{(j)}$ will be either a real-valued function (where the need for approximation comes from the discretized solution of differential equations that define these functions), an identity operator (which will be approximated by interpolation operators based on finitely many deterministic samples), or an expected value (which is again approximated using finitely many, either deterministic or random, samples). The multilinear map $\bl$ will be a combination of applications of operators to elements of their domain and of tensor products of operators. Finally, the value $v$ will be a scalar, a real-valued function, or an operator.

A straightforward way to approximate $\goal$ is to consider 
\begin{equation*}
	\bl(\app_{\ns^{(1)}}^{(1)},\dots,\app_{\ns^{(\dt)}}^{(\dt)})
\end{equation*} with large $\ns^{(j)}$ for all $j\in\{1,\dots,\dt\}$.
Under the assumptions stated below, the work required by this approach for an error of size $\epsilon>0$ grows like $\epsilon^{-\gamma_1/\beta_1-\dots-\gamma_n/\beta_n}$. 
We will derive an alternative, decomposition based approximation of $\goal$ that reduces the workload to $\epsilon^{-\rho}$, $\rho:=\max_{j=1}^n\gamma_j/\beta_j$, up to possible logarithmic factors. In the context of integration and interpolation problems, this approach is known as Smolyak's algorithm \cite{smolyak1963quadrature}.
\begin{itemize}
\item \textbf{Assumption 1 (Continuity):} The map $\bl$ is continuous. This is equivalent\fxnote{(\cite[Proposition 4.1]{AmannEscher2008})} to the existence of a constant $C>0$ such that 
\begin{equation*}
\label{eq:cross}
\|\bl(a^{(1)},\dots,a^{(\dt)})\|\leq C \prod_{j=1}^n\|a^{(j)}\|
\end{equation*}
for any $a^{(j)} \in  \factor^{(j)}, j\in\{1,\dots,\dt\}$.\fxnote{According to Lemma 4.33 of \cite{Hackbusch2012} this is always true for a norm that is defined on the whole space, if all factors are Banach (except one)}
 Here and in the remainder of this work, we use the generic symbol $\|\cdot\|$ to denote norms whenever it is evident from the context which specific norm is meant. 
\end{itemize}

 \begin{itemize}
 \item \textbf{Assumption 2 (Componentwise approximability):} For each $j\in\{1,\dots,\dt\}$, we have
 \begin{equation*}
 \|\inp^{(j)}-\app^{(j)}_{\ns}\|\lesssim_{\ns}  \ns^{-\beta_j}
 \end{equation*}
  for some $\beta_j>0$. We use the notation $\lesssim_{\ns}$ to denote inequalities that hold up to a factor that is \emph{independent} of $\ns$.
 \item \textbf{Assumption 3 (Componentwise required work):} For each $j\in\{1,\dots,\dt\}$, the construction of $\app^{(j)}_{\ns}$ requires the work 
 \begin{equation*}
 \work(\app^{(j)}_{\ns})\lesssim_{N}\ns^{\gamma_j}
\end{equation*} for some $\gamma_j>0$.
 \item \textbf{Assumption 4 (Overall work):} The work required for the evaluation of $\bl$ is subadditive and multiplicative,
\begin{align*} \work(\bl(a^{(1)},\dots,a^{(\dt)})+\bl(b^{(1)},\dots,b^{(\dt)}))&\leq \work(\bl(a^{(1)},\dots,a^{(\dt)}))+\work(\bl(b^{(1)},\dots,b^{(\dt)}))\\
\work(\bl(a^{(1)},\dots,a^{(\dt)}))&=\prod_{j=1}^\dt \work(a^{(j)})
\end{align*}
 for any $a^{(j)}, b^{(j)}\in\factor^{(j)}$, $j\in\{1,\dots,\dt\}$.
 \end{itemize}
To define Smolyak's algorithm, for any $j\in\{1,\dots,\dt\}$ we consider subsequences $\ns^{(j)}_l:=\exp(t_j l)$, $l\in\Np:=\{1,\dots\}$, with $t_j>0$ to be chosen below, and we define the consecutive differences
\begin{align*}
&\Delta^{(j)}_{l}:=\app_{\ns^{(j)}_{l}}^{(j)}-\app_{\ns^{(j)}_{l-1}}^{(j)}\quad\forall l\geq 1
 \end{align*}
 with the auxiliary definition $\app_{\ns^{(j)}_{0}}^{(j)}:=0$.
 Deferring questions of convergence to \Cref{pro:convergence} below, we can write
 \begin{equation}
 \label{eq:inftensor}
 \begin{split}
 \goal&=\bl(\inp^{(1)},\dots,\inp^{(\dt)})\\
 &=\bl(\sum_{l_1=1}^{\infty}\Delta^{(1)}_{l_1},\dots,\sum_{l_{\dt}=1}^\infty \Delta^{(\dt)}_{l_{\dt}})\\
 &=\sum_{\bm{l}\in \Np^\dt}\bl(\Delta^{(1)}_{l_1},\dots,\Delta^{(\dt)}_{l_\dt})\\
 &=:\sum_{\bm{l}\in \Np^\dt}\Delta_{\bm{l}}.
 \end{split}
 \end{equation}
 It is now reasonable to restrict the final sum in the above decomposition of $v$ to those multi-indices $\bm{l}\in\Np^\dt$ for which the ratio of work and contribution (measured by the norm) associated with 
 \begin{equation*}
 \Delta_{\bm{l}}=\bl(\Delta^{(1)}_{l_{1}},\dots,\Delta^{(\dt)}_{l_{\dt}})
 \end{equation*}
 is below some threshold. Thanks to Assumptions 3 and 4, the work associated with $\Delta_{\bm{l}}$ can be bounded by
  \begin{equation}
  \label{eq:workl}
  \begin{split}
 \work(\Delta_{\bm{l}})&=\prod_{j=1}^\dt \work(\Delta_{l_j}^{(j)})\\
 &\leq \prod_{j=1}^\dt\left[ \exp(\gamma_jt_jl_j)+\exp(\gamma_jt_j(l_{j}-1))\right]\\
 & \lesssim_{\bm{l}} \prod_{j=1}^\dt \exp(\gamma_jt_jl_j)
 \end{split}
  \end{equation}
  and due to Assumptions 1 and 2, the norm of $\Delta_{\bm{l}}$ can be bounded by
  \begin{equation}
  \label{eq:norml}
  \begin{split}
  \|\Delta_{\bm{l}}\|&\leq C\prod_{j=1}^\dt \|\app^{(j)}_{l_j}-\app^{(j)}_{l_j-1}\|\\
  &\lesssim_{\bm{l}} \prod_{j=1}^\dt \exp(-\beta_jt_jl_j).
  \end{split}
  \end{equation}
 Therefore, we approximate the work-to-contribution ratio of $\Delta_{\bm{l}}$ by
 \begin{equation}
 \label{eq:ratio}
\exp(\sum_{j=1}^\dt (\gamma_j+\beta_j)t_jl_j).
 \end{equation}
 Since strict inequalities in our derivations above are possible, this approximation may not be exact, and therefore merely functions as a motivation for the following definitions. Looking at \eqref{eq:ratio}, we may proceed in two ways\fxnote{which yield the same rates of convergence}. Either we choose $t_j:=1/(\gamma_j+\beta_j)$ and restrict the sum in \Cref{eq:inftensor} to those $\Delta_{\bm{l}}$ with $|\bm{l}|_1:=l_{1}+\dots +l_{\dt}\leq L\in\Np$ or we take $t_j$ to be constant and sum up all $\Delta_{\bm{l}}$ with $(\bm{\gamma}+\bm{\beta})\cdot\bm{l}\leq L$. We choose the first option and define \emph{Smolyak's algorithm}
 \begin{equation}
 \label{def:smolyak}
 \smol:=\sum_{|\bm{l}|_1\leq L} \Delta_{\bm{l}}=\sum_{|\bm{l}|_1\leq L} \bl(\Delta^{(1)}_{l_1},\dots,\Delta^{(\dt)}_{l_\dt}).
 \end{equation}
  The following \emph{combination rule} \cite{GriebelSchneiderZenger1992} can be proven verbatim as in \cite[Lemma 1]{Wasilkowski1995} and can facilitate numerical implementations:
  \begin{equation}
  \label{eq:combinationrule}
  \smol=\sum_{L-\dt+1\leq |\bm{l}|_1\leq L} (-1)^{L-|\bm{l}|_1}\binom{\dt-1}{L-|\bm{l}|_1} \goal_{\bm{l}},
  \end{equation}
  where 
  \begin{equation*}
  \goal_{\bm{l}}:=\bl(\app^{(1)}_{\ns^{(1)}_{l_1}},\dots,\app^{(\dt)}_{\ns^{(\dt)}_{l_\dt}}).
  \end{equation*}
Of course, $\smol$ is simply an  element of $\vsgoal$; the word \emph{algorithm} is used because in Smolyak's original publication \cite{smolyak1963quadrature} the factors $\app^{(j)}_{\ns}$ were univariate interpolation or quadrature formulas, the multilinear map corresponded to the tensor product of these operators, and Smolyak's algorithm provided instructions for the combination of the previously known univariate formulas to obtain novel corresponding multivariate formulas.

 By \Cref{eq:workl}, the work associated to $\smol$ is bounded by 
\begin{equation}
\work(\smol)\leq \sum_{|\bm{l}|_1\leq L} \exp(\bm{g}\cdot \bm{l}),
\end{equation}
where $\bm{g}=(g_1,\dots,g_{\dt})$ with $g_j:=\gamma_j/(\gamma_j+\beta_j)$. The exponential sum on the right hand side of the previous inequality  been estimated in \cite[Lemma 6]{Haji-AliNobileTamelliniEtAl2015a} with the result 
\begin{equation}
\label{eq:work}
\work(\smol)\lesssim_{L} \exp(g_{\max} L)L^{\dt^{+}(\bm{g})-1},
\end{equation}
where $g_{\max}:=\max_{j=1}^\dt g_j$ and $\dt^{+}(\bm{g}):=|\{j: g_j=g_{\max}\}|$.
Furthermore, \Cref{eq:norml,eq:inftensor} (see \Cref{pro:convergence} below for a rigorous justification) show that 
\begin{equation}
\label{eq:worksum}
\|\goal-\smol\|=\|\sum_{|\bm{l}|_1>L}\Delta_{\bm{l}}\|\lesssim_{L} \sum_{|\bm{l}|_1>L}\exp(-\bm{b}\cdot\bm{l}),
\end{equation}
where $\bm{b}=(b_1,\dots,b_\dt)$ with $b_j:={\beta_j}/({\gamma_j}+{\beta_j})=1-{g_j}$. Again, it remains to bound an exponential sum; although this time an infinite one with decaying terms. This has been done in \cite[Lemma 7]{Haji-AliNobileTamelliniEtAl2015a}, with the result
 \begin{equation}
 \label{eq:workestimsum}
  \|\goal - \smol\|\lesssim_{L} \exp(-{b}_{\min}{L})L^{\dt^{-}(\bm{b})-1},
 \end{equation}
where $b_{\min}:=\min_{j=1}^\dt b_j$ and $\dt^{-}(\bm{b}):=|\{j: b_j=b_{\min}\}|$.  
\\

To summarize the results in a succinct fashion, we define
\begin{equation*}
\rho:=\max_{j=1}^{\dt}\gamma_j/\beta_j=g_{\max}/b_{\min}
\end{equation*}
and
\begin{equation*}
\dt_0:=|\{j:\gamma_j/\beta_j=\rho\}|=\dt^{-}(\bm{b})=\dt^{+}(\bm{g}).
\end{equation*} 

\begin{thm}\textbf{(Convergence of sparse approximations)}
\label{thm:central}
For $\epsilon>0$ small enough, we can choose $L=L(\epsilon)$ such that 
\begin{equation}
\label{eq:centralerror}
\|\smol-\goal\|\lesssim_{\epsilon} \epsilon
\end{equation}
and
\begin{align}
\label{eq:epswork}
\work(\smol)\lesssim_{\epsilon}\epsilon^{-\rho}|\log\epsilon|^{(\dt_0-1)(1+\rho)}.
\end{align}
\end{thm}
\begin{proof}
Given $\epsilon>0$, let $L$ be the largest integer such that 
$\psi(L):=\exp(-{b}_{\min}{L})L^{\dt_0-1}>\epsilon$. By \Cref{eq:workestimsum}, and because $\psi(L)/\psi(L+1)$ is bounded and $\psi(L+1)\leq \epsilon$, this implies \Cref{eq:centralerror}. Furthermore, by \Cref{eq:work} we have
\begin{equation*}
\begin{split}
\work(\smol)&\lesssim_{L}  \exp(g_{\max}L)L^{\dt_0-1}\\
&= \exp(\rho b_{\min}L)L^{-(\dt_0-1)\rho}L^{(\dt_0-1)(1+\rho)}\\
&\lesssim_{L} \psi(L)^{-\rho}|\log\psi(L)|^{(\dt_0-1)(1+\rho)}\\
&\leq \epsilon^{-\rho}|\log\epsilon|^{(\dt_0-1)(1+\rho)},
\end{split}
\end{equation*}
where the last inequality holds for all small enough $\epsilon>0$ by the choice of $L$.
\end{proof}
\begin{rem}
	\Cref{thm:central} generalizes results on sparse wavelet approximation that were proven in \cite{griebel2013note} using orthogonal decompositions. Indeed, we show in \Cref {sec:kernel} how high-dimensional approximation can be analyzed as multilinear approximation problem and deduce results similar to those in \cite{griebel2013note} but for kernel-based approximation. 
\end{rem}
\begin{rem}\textbf{(Exponential convergence)} It may happen that one of the inputs exhibits exponential convergence, 
 \begin{equation*}
\|\inp^{(j)}-\app^{(j)}_{\ns}\|\lesssim_{\ns} \exp(-s_j\ns)
 \end{equation*} and algebraic work, 
 \begin{equation*}
 \work(\app^{(j)}_{\ns})\lesssim_{\ns} \ns^{\gamma_j}.
 \end{equation*}
Such inputs satisfy Assumption 2 for any exponent and thus can always be added to a problem without increasing the bounds in \Cref{thm:central}. When all inputs converge exponentially, improved exponential convergence rates can be obtained by an extended analysis, see \Cref{diss} and \cite{GriebelOettershagen}.
\end{rem}
\begin{rem}\textbf{(Logarithmic factors)}
If the work of one of the inputs exhibits additional logarithmic factors,
\begin{equation*}
\work(\app^{(j)}_{\ns})\lesssim_{\ns} \ns^{\gamma_j}(\log \ns)^{\mu_j},
\end{equation*}
then the work required for an error of size $\epsilon$ increases by the factor $|\log \epsilon|^{s}$, where $s:=\sum_{\{j:\gamma_j/\beta_j=\rho\}} \mu_j$. Indeed, $\sum_{|\mathbf{l}|_1\leq L}\exp(\mathbf{g}\cdot\mathbf{l})\prod_{j=1}^{\dt}l_j^{\mu_j}\lesssim_{L}\exp(g_{\max}L)L^{n^+(\mathbf{g})-1+s}$, which follows from a simple supremum bound together with \cite[Lemma 6]{Haji-AliNobileTamelliniEtAl2015a}. Therefore, \Cref{eq:work} holds with the additional factor $L^s$ and \Cref{thm:central} holds with the additional factor $|\log \epsilon|^{s}$.
\end{rem}
\begin{rem}\textbf{(Tracking constants)}
Provided more explicit bounds on error and work, 
 \begin{align*}
 \|\inp^{(j)}-\app^{(j)}_{\ns}\|\leq C_{E,j} \ns^{-\beta_j}\\
 \work(\app^{(j)}_{\ns})\leq C_{W,j}\ns^{\gamma_j},
\end{align*}
we may refine the work bound in \Cref{thm:central} to
\begin{align}
\work(\smol)\lesssim_{\epsilon,C_{E},C_{W}}C_{W}C_{E}^{\rho}\epsilon^{-\rho}|\log\epsilon|^{(\dt_0-1)(1+\rho)},
\end{align}
for $C_{W}:=\prod_{j=1}^\dt C_{W,j}$ and $C_{E}:=\prod_{j=1}^\dt C_{E,j}$.
\end{rem}
\begin{rem}\textbf{(Integer constraints)} We assumed that we can choose $\ns\in\Rnn:=\{x\in \R : x\geq 0\}$. In practice, $\ns$ is often restricted to being a natural number. If we implicitly round up all occurences of $\ns$, then the analysis above goes through since the error bound in \Cref{eq:norml} persists unaltered and the work bound in \Cref{eq:workl} persists with another constant.
\end{rem}

\begin{lem}
\label{pro:convergence}
Under Assumptions 1 and 2, the elements $\Delta_{\bm{l}}\in \vsgoal$, $\bm{l}\in\N^{\dt}$ are absolutely summable and their sum is $\goal$. 
In particular, 
\begin{equation*}
\goal-\sum_{|\bm{l}|_1\leq L}\Delta_{\bm{l}}=\sum_{|\bm{l}|_1>L} \Delta_{\bm{l}}.
\end{equation*}
\end{lem}
\begin{proof}
By \Cref{eq:norml} and \cite[Lemma 7]{Haji-AliNobileTamelliniEtAl2015a} we have
\begin{equation*}
\begin{split}
\sum_{\bm{l}\in\N^d}\|\Delta_{\bm{l}}\|<\infty.
\end{split}
\end{equation*}
Therefore, all rearrangements yield the same limit, if one exists.
But
\begin{equation*}
\begin{split}
\sum_{|\bm{l}|_{\infty}\leq L}\Delta_{\bm{l}}=\bl( \sum_{l_1=1}^{L}\Delta^{(1)}_{l_1},\dots,\sum_{l_{\dt}=1}^L \Delta^{(\dt)}_{l_{\dt}})=\bl(\app^{(1)}_{L},\dots,\app^{(\dt)}_{L})\to \goal
\end{split}
\end{equation*}
by continuity of $\bl$.
\end{proof}

 \section{Kernel-based approximation}
 \label{sec:kernel}

 In this section, we describe kernel-based approximation methods \cite{SchabackWendland2006,wendland2004scattered}, which we later use to approximate response surfaces. 
 
  Assume we want to reconstruct an element $f$ of a Hilbert space $(H,\langle\cdot,\cdot\rangle)$ from the output $Tf$ of a linear \emph{sampling operator} $T\in\mathcal{L}(H,\R^\ns)$, with $N\in \Np$. 
  
  If $H$ is infinite-dimensional, then the output of $T$ never uniquely determines an element of $H$. To resolve this ambiguity we select the interpolant with minimal norm,
 \begin{align}
 \label{eq:app}
 Sf:=\argmin_{\stackrel{s\in H}{Ts=Tf}}\|s\|^2,
 \end{align} 
 where $\|s\|^2:=\langle s,s\rangle$.
In the remainder of this work, we refer to $S$ as the \emph{best-approximation associated with T}, which is justified by property (ii) in \Cref{thm:opt} below.

We denote by $T^*\colon \mathbb{R}^\ns\to H$ the Hilbert space adjoint of $T$. For the sake of simplicity we assume that $H$ is a real Hilbert space and that $T$ is surjective, which is the case in all applications we consider in this work. In particular, this implies that $T^*$ is injective and $TT^*$ invertible.

\begin{thm}\label{thm:opt}
	\begin{enumerate}[(i)]
		\item 	 The best-approximation $S$ from \Cref{eq:app} is well-defined, linear, and satisfies
		\begin{align}\label{eq:exint}
		S=T^*(TT^*)^{-1}T.
		\end{align} 

		\item 
		$Sf$ is the best approximation to $f$ from $(\ker T)^{\bot}$: For any $f\in H$, we have
	\begin{align*}
	\|f-Sf\|=\inf_{g\in (\ker T)^{\bot}}\|f-g\|.
	\end{align*}

	\end{enumerate}
\end{thm}
\begin{proof}
\Cref{eq:app} defines $Sf$ as the minimal-norm approximation of $0$ from the affine subspace $f+\ker T$. In Hilbert spaces, this coincides with the orthogonal projection, which is uniquely determined by $TSf=Tf$ and $\langle Sf,v\rangle =0$ for all $v\in \ker T$. Both of these equations are satisfied by $Sf=T^*(TT^*)^{-1}Tf$, which proves (i). From here we see that $Sf\in \I T^*=(\ker T)^{\bot}$ and $f-Sf\in \ker T=((\ker T)^{\bot})^{\bot}$, which implies that $Sf$ is also the orthogonal projection of $f$ onto $(\ker T)^{\bot}$ and thus (ii) holds.
\end{proof}
If our objects of interest are real-valued functions on domains $\dom\subset\R^d$, an explicit classification of all Hilbert spaces of such functions is desirable. Under the additional assumption that point evaluations are continuous, this can be achieved through \emph{reproducing kernels}.
\fxnote{Wendland tut so als ob der reproduzierende kernel automatisch stetig waere, das ist er aber nicht. es gilt: kernel stetig<-> alle funktionen stetig}
\begin{definition}
A Hilbert space $(H,\|\cdot\|)$ of functions on $\dom\subset\R^d$ such that all point evaluations $\delta_x f:=f(x)$, $x\in\dom$ are continuous with respect to $\|\cdot\|$ is called \emph{Reproducing Kernel Hilbert Space} (RKHS). We call $\Phi\colon\dom\times\dom\to\R$, $\Phi(x,y):=\Riesz(\delta_x)(y)$ the \emph{reproducing kernel} of $H$, where $\Riesz\colon H^*\to H$ is the Riesz isometry. 
\end{definition}
It can be shown \cite{Aronszajn1950} that the reproducing kernel of an RKHS is symmetric and positive definite on $\dom\times\dom$, meaning that
\begin{equation*}
(\Phi(x_i,x_j))_{i,j=1}^\ns\fxnote{oder strikt positive defint?}
\end{equation*}
is positive definite for any $\{x_1,\dots,x_\ns\}\subset\dom$. Conversely, any function that satisfies these conditions is the reproducing kernel of a unique RKHS, which is called the \emph{native space} of $\Phi$ and denoted by $\native{\Phi}{\dom}$. 

Depending on the context, the application of best-approximation theory in Hilbert spaces to the reconstruction of functions in RKHS is called kernel-based approximation \cite{FasshauerMcCourt2016}, scattered data approximation \cite{wendland2004scattered}, kriging \cite{Stein2012}, or kernel learning \cite{ScholkopfSmola2001}.

For practical applications it is crucial that, given the reproducing kernel, best-approximations can be computed exactly. The following proposition is well known in the theory of kernel-based approximation; we provide a proof for the convenience of the reader.
\begin{pro}\label{pro:pro}
Let $T=(\lambda_1,\dots,\lambda_\ns)$ with $\lambda_i\in H^*$. Then 

\begin{equation*}
Sf=\sum_{i=1}^{\ns} s_{i}\lambda_if,
\end{equation*}
where 
\begin{equation*}
s_i=\sum_{j=1}^n u_{ij}\lambda_j^1\Phi\in H
\end{equation*}
and $(u_{ij})_{i,j=1}^{\ns}$ are the entries of the inverse of $(\lambda_{i}^2\lambda_{j}^1\Phi)_{i,j=1}^{\ns}$.
Here, the superscript indicates the variable that is acted on. For example, if $T$ consists of point evaluations in $\{x_1,\dots,x_\ns\}\subset \dom$, then $\lambda_j^1\Phi=\Phi(x_j,\cdot)\in H$ and  $\lambda_{i}^{2}\lambda_{j}^{1}\Phi=\Phi(x_j,x_i)\in\R$.
\end{pro}
\begin{proof}
In view of \Cref{eq:exint} it suffices to note that $T^*v=\sum_{i=1}^\ns \lambda_{i}^{1}\Phi v_i$ for any $v\in\R^\ns$ and therefore $TT^*v=(\sum_{i=1}^\ns  (\lambda_{j}^{2}\lambda_{i}^{1}\Phi)v_i)_{j=1}^\ns$. We have seen in \Cref{thm:opt} that $TT^*$ is invertible. 
\end{proof}

Maybe the most common examples of reproducing kernel Hilbert spaces are the isotropic Sobolev spaces $H^{\reg}(\R^d)$ with $\reg>d/2$, for which point evaluations are continuous by Sobolev's embedding theorem.
We give here a general definition using Fourier transforms that is suited to our interest in multilinear problems. 
For a partition $D=\{D_j\}_{j=1}^n$ of $\{1,\dots,d\}$, with $d_j:=|\pa_j|$, and $\bm{\reg}\in\Rnn^\dt$, we define 
\begin{equation*}
H^{\bm{\reg}}_{\pa}(\R^d):=\{f\in L^2(\R^d) : \|f\|_{H^{\bm{\reg}}_{\pa}(\R^d)}:=\|\prod_{j=1}^\dt (1+|\omega_{\pa_j}|^2)^{\reg_j/2}\Fourier{f} \|_{L^2(\R^d)}<\infty)\},
\end{equation*}
equipped with the obvious inner product, 
where we denote by $|\cdot|_2$ the Euclidean norm, and where we access groups of components of $\omega\in\R^d$ by the subscripts $\pa_j$. Furthermore, for any $\dom\subset \R^d$ we define $H^{\bm{\beta}}_{\pa}(\dom)$ as space of restrictions of functions in $H^{\bm{\reg}}_{\pa}(\R^d)$ with the norm $\|f\|_{H^{\bm{\reg}}_{\pa}(\dom)}:=\inf_{g|_{\dom}=f}\|g\|_{H^{\reg}_{\pa}(\R^d)}$.

Special cases are the isotropic Sobolev spaces, which correspond to $\pa_1=\{1,\dots,d\}$ and $\bm{\reg}=\reg\in\Rnn$, and the Sobolev spaces $H^{\reg}_{\mix}(\R^d)$ of dominating mixed smoothness, which correspond to $\pa_j=\{j\}$  and $\beta_j=\beta\in\Rnn$ for $j\in\{1,\dots,d\}$.

 If $\bm{\reg}\in\Nz^\dt$, then we have the characterization
\begin{equation*}
H^{\bm{\reg}}_{\pa}(\R^d):=\{f\in L^2(\R^d): \forall \bm{\regs}\in\Nz^{d},  |\bm{\alpha}_{\pa_j}|_{1}\leq \reg_j, j\in\{1,\dots,\dt\}: \|\partial_{\bm{\alpha}}f\|_{L^2(\R^d)}<\infty \}.
\end{equation*}

 Furthermore, we see below that the spaces $H^{\bm{\reg}}_D(\Gamma)$ are tensor products of isotropic Hilbert spaces when $\Gamma$ is a product domain. The multilinearity of the tensor product will later allow us to apply the general framework from \Cref{sec:sparse}.
 \begin{definition}[{\cite{Hackbusch2012}}]
 	\label{def:hilbert}
 	If $(H^{(j)},\langle \cdot,\cdot\rangle_{j})$ are Hilbert spaces for $j\in\{1,\dots,\dt\}$, then the unique bilinear extension of 
 	\begin{equation*}
 	\langle \bigotimes_{j=1}^\dt f^{(j)}_{1},\bigotimes_{j=1}^\dt f^{(j)}_{2}\rangle_{H}:=\prod_{j=1}^\dt\langle f^{(j)}_{1},f^{(j)}_{2}\rangle_{j}\quad \forall\; f_i^{(j)}\in H^{(j)}, i\in\{1,2\}, j\in\{1,\dots,\dt\}
 	\end{equation*}
 	is an inner product on their algebraic tensor product. We call the completion of the algebraic tensor product under this inner product \emph{the Hilbert tensor product} and denote it $\bigotimes_{j=1}^\dt H^{(j)}.$
 \end{definition}

\begin{pro}
\label{pro:tensornative}
\begin{enumerate}[(i)]
\item
Let $\{\pa_j\}_{j=1}^\dt$  be a partition of $\{1,\dots,d\}$. If for each $j\in\{1,\dots,n\}$ we have a function $\phi_j\colon\R^{d_j}\to\R$ that satisfies
\begin{equation*}
(1+|\omega|_2^2)^{-\reg_j}\lesssim_{\omega} |\Fourier{\phi_j}(\omega)| \lesssim_{\omega}  (1+|\omega|_2^2)^{-\reg_j},
\end{equation*} 
for some $\reg_j/2>d_j$, then $\Phi(x,y):=\prod_{j=1}^\dt\phi_j(x_{\pa_j}-y_{\pa_j})\colon\R^d\times\R^d\to\R$ is a reproducing kernel with native space $\native{\Phi}{\R^d}\simeq H^{\bm{\reg}}_{\pa}(\R^d)$.

\item If $\Phi\colon \supdom\times \supdom\to\R$ is a reproducing kernel, and $\dom\subset\supdom$, then $\Phi|_{\dom\times \dom}$ is the reproducing kernel of the space of restrictions equipped with the natural norm: 
\begin{equation*}
\native{\Phi|_{\dom\times \dom}}{\dom}=(\{g|_{\dom}: g\in\native{\Phi}{\supdom}\}, \|f\|_{\native{\Phi}{\dom}}=\inf_{g|_{\dom}=f}\|g\|_{\native{\Phi}{\supdom}}).
\end{equation*}
\fxnote{No topologic or geometric restrictions on $\supdom$ and $\dom$ are required.}
\item If $\Phi_j\colon \dom_j\times\dom_j\to\R$ are reproducing kernels for $j\in\{1,\dots,n\}$, then $\Phi:=\bigotimes_{j=1}^\dt\Phi_j$ is a reproducing kernel and $\native{\Phi}{\prod_{j=1}^\dt\dom_j}=\bigotimes_{j=1}^\dt\native{\Phi_j}{\dom_j}$. 
\item Let $\{\pa_j\}_{j=1}^\dt$ be a partition of $\{1,\dots,d\}$ and assume that $\reg_j>d_j/2$ for all $j\in\{1,\dots,\dt\}$.Then 
\begin{equation*}
H^{\bm{\reg}}_{\pa}(\prod_{j=1}^\dt \dom_j)=\bigotimes_{j=1}^\dt H^{\reg_j}(\dom_j).
\end{equation*}
for any $\dom_j\subset\R^{d_j}$, $j\in\{1,\dots,\dt\}$.
\end{enumerate}
\end{pro}
\begin{proof}
\begin{enumerate}[(i)]
\item Follows from \cite[Theorem 10.12]{wendland2004scattered}.
\item This is \cite[Section 5, Theorem 1]{Aronszajn1950}.
\item This is \cite[Section 8, Theorem 1]{Aronszajn1950}.
\item Follows from combining (i) through (iii).
\end{enumerate}
\end{proof}

 A family of functions that satisfy the condition in part (i) are the \emph{Matérn functions} 
 \begin{equation*}
 \phi\colon\R^d\to \R, x\mapsto \phi(x):=\frac{2^{1-\reg}}{\Gamma(\reg)}|x|_2^{\reg-d/2}K_{\reg-d/2}(|x|_2),
 \end{equation*} where $\reg>d/2$, and $K_{\reg-d/2}(r)$ is the modified Bessel function of the second kind of order $\reg-d/2$. Their Fourier transform \emph{equals} $(1+|\omega|_2^2)^{-\reg}$ \cite[Theorem 6.13]{wendland2004scattered}. By parts (ii) through (iii) of the previous proposition, we therefore have explicit expressions of the reproducing kernels of the generalized Sobolev spaces $H^{\bm{\reg}}_{\pa}(\dom)$. In combination with \Cref{pro:pro} this allows the effective computation of best-approximations in these spaces.
 
Error bounds for best-approximations in isotropic Sobolev spaces can be deduced from the \emph{sampling inequality} in \Cref{pro:sampling} below.
For subsets $X\subset\dom$ we denote by 
\begin{equation*}
h_{X,\dom}:=\sup_{y\in \dom}\inf_{x\in X}|x-y|_2
\end{equation*}
the \emph{fill-distance} of $X$ in $\dom$.
 \begin{pro}[{\cite[Theorem 2.6]{WendlandRieger2005}}]
 \label{pro:sampling}
 Let $\dom\subset\R^d$ be a bounded Lipschitz domain. There exists $h_0>0$ such that, for any finite set $X\subset \dom$ with $h_{X,\dom}<h_0$ and any $r\in H^{\reg}(\dom)$, we have
 \begin{equation*}
 \|r\|_{L^2(\dom)}\lesssim_{h,r} h_{X,\dom}^{\reg}\|r\|_{H^{\reg}(\dom)}+\max_{x\in X}|r(x)|
 \end{equation*}
 \noproof
 \end{pro}
 To turn the previous proposition, which is a purely theoretical property of Sobolev functions, into a convergence result for kernel-based interpolation, all that is needed is stability of kernel-based interpolation, $\|S_{X}u\|_{H^{\beta}(\domPS)}\leq \|u\|_{H^{\beta}}$, which follows directly from the definition in \MakeUppercase Equation\nobreakspace \textup {(\ref {eq:app})}. In the following proposition and in the remainder of this work, we denote by
 \begin{equation*}
 \|A\|_{H\to G}:=\sup_{\|h\|_{H}=1}\|Ah\|_{G}
 \end{equation*}
 the operator norm of a linear operator $A\colon H\to G$ between general normed vector spaces $H$ and $G$.
 \begin{pro}
 \label{pro:sobolev}
 Let $\dom\subset\mathbb{R}^d$ be a bounded Lipschitz domain, and let $\Phi$ be a reproducing kernel such that $\native{\Phi}{\dom}\simeq H^{\beta}(\dom)$. For any $0\leq\alpha\leq  \beta $ we have
 \begin{equation*}
 \|\Id-S_{X}\|_{H^\beta(\dom)\to H^{\alpha}(\dom)}\lesssim_{X} h_{X,\dom}^{\beta-\alpha},
 \end{equation*}
 where $S_{X}$ is the best-approximation in $H^{\beta}(\dom)$ associated to point evaluations in $X\subset\dom$.
 \end{pro}
 \begin{proof}
By the Gagliardo-Nirenberg interpolation inequality \cite[Theorem 1]{Nirenberg1966}, it suffices to consider the cases $\alpha=0$ and $\alpha=\beta$. For the first case, consider $u\in H^{\beta}(\dom)$ and apply \Cref{pro:sampling} to $r:=u-S_Xu$. The claim follows because $r|_{X}\equiv 0$ and \begin{equation*}\|r\|_{H^\reg(\dom)}\leq \|u\|_{H^{\reg}(\dom)}+\|Su\|_{H^\reg(\dom)}\leq 2\|u\|_{H^{\reg}(\dom)}\end{equation*} by the definition of $Su$. The second case follows directly from the previous inequality.
 \end{proof}

We now consider the sparse approximation of functions in a native space $\native{\Phi}{\dom}$ on a product domain $\dom:=\prod_{j=1}^\dt\dom_j$ with a tensor product kernel $\Phi:=\Phi_1\otimes\dots\otimes \Phi_\dt\colon \dom\times \dom\to \R$.

Since
\begin{equation*}
\Id_{\native{\Phi}{\dom}}=\Id_{\native{\Phi_1}{\dom_1}}\otimes\dots\otimes\Id_{\native{\Phi_\dt}{\dom_\dt}},
\end{equation*}
 we may apply Smolyak's algorithm to the multilinear tensor product of operators, and we 
 obtain an approximation of the identity operator $\Id_{\native{\Phi}{\dom}}$ that employs point evaluations in a \emph{sparse grid} \cite{smolyak1963quadrature,NovakRitter1996,GerstnerGriebel1998}.
 
 Before we become more specific, let us recall the tensor product of operators.
 \begin{pro}[{\cite[Proposition 4.127]{Hackbusch2012}}]
 		\label{pro:cross}
 		Let  $H^{(j)}$ and $G^{(j)}$, with $j\in\{1,\dots,\dt\}$,  be Hilbert spaces and denote by $H=\bigotimes_{j=1}^\dt H^{(j)}$ and $G=\bigotimes_{j=1}^\dt G^{(j)}$ their Hilbert tensor products. Given operators  $A^{(j)}\in  \mathcal{L}(H^{(j)},G^{(j)})$, $j\in\{1,\dots,\dt\}$  we define the tensor product operator on 
 		the algebraic tensor product of the spaces $H^{(j)}, j\in\{1,\dots,\dt\}$ by 
 		\begin{equation*}
 	\Big(	\bigotimes_{j=1}^\dt A^{(j)}\Big)\Big(\bigotimes_{j=1}^\dt h^{(j)}\Big):=\bigotimes_{j=1}^n A^{(j)}h^{(j)}
 		\end{equation*}
 		and by multilinear extension.
 		This \emph{algebraic} tensor product operator satisfies
 		\begin{equation*}
 		\|\bigotimes_{j=1}^\dt A^{(j)}\|_{H\to G}=\prod_{j=1}^\dt \|A^{(j)}\|_{H^{(j)}\to G^{(j)}}
 		\end{equation*}
 		on its domain of definition and can therefore be extended to the completion $H$ of the algebraic
 		tensor product, maintaining the same bound on the operator norm.
 		\noproof
 	\end{pro}
 	\fxnote{Wenn \cite[Um Lemma 4.33 rum]{Hackbusch2012} recht hat gilt das immer, solange die operator norm endlich ist auf allen moeglichen tensorprodukten (dann definiert sie naemlich eine norm auf dem tensorprodukt der faktoroperatoren)}
 
 Now consider the case where the factors $\native{\Phi_j}{\dom_j}$ are isotropic Sobolev spaces $H^{\reg_j}(\dom_j)$, $\dom_j\subset\R^{d_j}$,
 and we apply Smolyak's algorithm to approximate
  \begin{equation*}
  v:=\bl(\Id_{H^{\reg_1}(\domPS_1)},\dots,\Id_{H^{\reg_\dt}(\domPS_\dt)}):=\Id_{H^{\reg_1}(\domPS_1)}\otimes\dots\otimes \Id_{H^{\reg_\dt}(\domPS_\dt)}=\Id_{H^{\bm{\reg}}_{D}(\domPS)}\in\mathcal{L}(H^{\bm{\reg}}_{D}(\domPS),H^{\bm{\regs}}_{D}(\domPS)).
  \end{equation*}
  Here, we consider $\Id_{H^{\bm{\reg}}_{D}(\domPS)}$ as taking values in $H^{\bm{\regs}}_{D}(\domPS)$ with $\bm{\regs}\leq \bm{\reg}$ in order to later obtain bounds in the norm of $H^{\bm{\regs}}_{D}(\domPS)$.
 To approximate the inputs $\Id_{H^{\reg_j}(\domPS_j)}$ we use best-approximations $S^{(j)}_N$ associated with point evaluations in sets $X^{(j)}_{\ns}\subset\dom_j$ with $|X^{(j)}_{\ns}|=\ns$ and $h_{X^{(j)}_{\ns},\dom_j}\lesssim_{\ns}\ns^{-1/d_j}$.

To be able to apply the general theory of \Cref{sec:sparse}, we need to verify the four assumptions from \Cref{sec:sparse}. Assumption 1 on the continuity of $\bl$ holds by \Cref{pro:cross} above.

Assumption 2 on the convergence of the input approximations has been established in \Cref{pro:sobolev}. Finally, to satisfy Assumptions 3 and 4 we assign as work to each point evaluation a unit cost, that is
\begin{equation*}
\work(S_{\ns^{(j)}}^{(j)}):=\ns^{(j)}  \quad\forall \ns^{(j)}>0, \,j\in\{1,\dots,\dt\}
\end{equation*}
and
\begin{equation*}
\work(S^{(1)}_{\ns^{(1)}}\otimes \dots \otimes S^{(\dt)}_{\ns^{(\dt)}}):=\prod_{j=1}^\dt \ns^{(j)}.
\end{equation*}
This is indeed the number of point evaluations required by the tensor product operator: By \Cref{pro:pro}, $S^{(j)}_{\ns^{(j)}}$ can be written as 
\begin{equation*}
\sum_{i=1}^{\ns^{(j)}} s^{(j)}_{i} \delta_{x^{(j)}_{i}}
\end{equation*}
for $x^{(j)}_{i}\in X^{(j)}_{\ns^{(j)}}$ and some $s^{(j)}_{i}\in \native{\Phi_j}{\dom_j}$, thus,
\begin{equation*}
S^{(1)}_{\ns^{(1)}}\otimes \dots \otimes S^{(\dt)}_{\ns^{(\dt)}}=\sum_{i_1=1}^{\ns^{(1)}}\dots \sum_{i_\dt=1}^{\ns^{(\dt)}}(\bigotimes_{j=1}^\dt s^{(j)}_{i_j}) \delta_{(x^{(1)}_{i_1},\dots,x^{(n)}_{i_n})}.
\end{equation*}
Therefore, \Cref{thm:spkernel} below follows directly from \Cref{thm:central}.

 \begin{thm}\textbf{(Sparse kernel-based approximations)}
 \label{thm:spkernel}
Assume $0\leq \bm{\regs}< \bm{\reg}$. For small enough $\epsilon$, we can choose $L$ such that Smolyak's algorithm with threshold $L$ satisfies
 \begin{equation*}
 \|\Id - \smoll_L(\Id)\|_{H^{\bm{\reg}}_{\pa}(\dom)\to H^{\bm{\regs}}_{\pa}(\dom)}\leq \epsilon
 \end{equation*}
and employs
 \begin{equation}
 \label{eq:kernelWork}
\ns_{\smoll_L(\Id)}\lesssim_{\epsilon} \epsilon^{-\rho}|\log\epsilon|^{(1+\rho)(\dt_0-1)}
 \end{equation}
  point evaluations in $\dom$, where
\begin{equation*}
\rho:=\textstyle{\max_{j=1}^\dt}d_j/(\reg_j-\regs_j)\text{ and }\dt_0:=|\{j:d_j/(\reg_j-\regs_j)=\rho\}|.
\end{equation*}\noproof
 \end{thm}

\begin{rem}\textbf{(Convergence in non-Hilbert norms)}
 The previous theorem is an extension of Theorem 4.41 in \cite{Schreiber2000}, which provides $L^\infty$ bounds for interpolation in an $n$-fold tensor product of a univariate Sobolev space.

To derive bounds on the error in non-Hilbert norms, such as $L^{\infty}$, observe that the proof of \Cref{pro:sobolev} goes through for non-Hilbert Sobolev spaces whose parameters satisfy certain conditions determined by the Gagliardo-Nirenberg inequality. It remains to be checked whether the desired norm is uniform with the native space norm in the sense introduced in \Cref{sec:sparse}. For bounds in the $L^{\infty}$-norm, one can simply use pointwise estimates and the fact that $\R=\R\otimes\dots\otimes \R$ is a Hilbert tensor product.
In this case, the result in \Cref{eq:kernelWork} holds true with 
\begin{equation*}
\rho:=\textstyle{\max_{j=1}^\dt}d_j/(\reg_j-d_j/2)\text{ and }\dt_0:=|\argmax_{j=1}^\dt\{d_j/(\reg_j-d_j/2)\}|.
\end{equation*}
 
\end{rem}
\begin{rem}\textbf{(Exponential convergence)}
	\label{diss} In \cite{GeorgoulisLevesleySubhan2013,DongGeorgoulisLevesleyEtAl2015}, Smolyak's algorithm is applied to approximation with Gaussian kernels (which have exponential univariate convergence rates). Theorem 4.44 in \cite{Schreiber2000} claims that using $h_l:=h_{X^{(j)}_l,\dom_j}\leq \exp(-l)$ yields exponential convergence in terms of the required samples,  $\exp(-c\ns)$, when Smolyak's algorithm is applied to the $n$-fold tensor product of univariate Gaussian kernel native spaces. However, the proof is based on the claim that the univariate interpolants satisfy 
\begin{equation}
\label{eq:faulty}
\|\Id-S^{(j)}_l\|\leq c_{\phi}h_{l}^k\quad \forall l,k\in\N
\end{equation}
 for constants $c_{\phi},C,c$ independent of $k$ and $l$, such that 
 \begin{equation}
 \label{eq:faulty2}
 c_{\phi}h_l^k\leq C\exp(-c h_{l}^{-1}) \quad\forall l,k\in\N,
\end{equation}
 which seems to be incorrect (consider $h_l<1$ and $k\to\infty$). Indeed the Smolyak algorithm with  $\ns^{(j)}_l\approx \exp(t_jl)$ is tailored to the situation of algebraically converging approximations and algebraically diverging work. To group contributions with equal work-to-error ratio as in \Cref{eq:ratio}, if the error converges exponentially and the work grows algebraically, one should use arithmetic subsequences $\ns^{(j)}_{l}\approx ml$. This yields only sub-exponential convergence, $\exp(-c \ns^{1/n})$, as does approximation with Gaussian kernels and quasi-uniform point sets. However, it can be shown that the factor $c=c(n)$ behaves better for large values of $n$ \cite{GriebelOettershagen}.
\end{rem}

\section{Applications to parametric and random PDEs}
\label{sec:UQ}
We apply Smolyak's algorithm to parametric partial differential equations of the form
\begin{equation}
\label{eq:PDE}
\begin{split}
\PDE_y(u_y)=f_y
\end{split}
\end{equation}
where both the partial differential operator  $\PDE_y$ and the right-hand side $f_y$ depend on a parameter $y\in\domPS\subset\mathbb{R}^d$.

 Assuming that there is a unique solution $u_y=\PDE_y^{-1}(f_y)$ for each $y\in\domPS$, our goal is to approximate the dependence of a scalar, possibly nonlinear, quantity of interest $\qoi(u_y)$ on the parameter $y$.

\subsection{Approximation of expectations}
\label{ssec:exp}
In this subsection, we assume that the parameter space $\domPS\subset\R^d$ is equipped with a probability distribution $\pi$, and our goal is to approximate expected values of quantities of interest.
We show how multilevel and multi-index methods can be regarded as applications of Smolyak's algorithm to generalized multilinear approximation problems. While the resulting methods are not new, we obtain streamlined proofs and see that linearity of the underlying PDE is not required to obtain multilinear approximation problems.

Our goal is to compute the expectation $E[\funct(\PDE_y^{-1}(f_y))]=\int \funct(\PDE_y^{-1}(f_y)) d\pi(y)$. 
 Roughly speaking, if suitable regularity results for the operators $\PDE_{y}$, $y\in\domPS$ (and possibly their linearizations) are available and if, for example, $Q$ is linear, then differentiation of Equation (\ref{eq:PDE}) with respect to $y$ shows that
 \begin{align*}
 \fq\colon\domPS&\to \R\\
 y&\mapsto q(y):=\qoi(\PDE_y^{-1}(f_y))
 \end{align*} 
 has similar differentiability properties with respect to $y$ as $\PDE_y$ and $f_y$.
 For rigorous results, consider for example \cite{Kuo2012,harbrecht2013multilevel,ChkifaCohenSchwab2015} or the discussion of our numerical experiment in Section \ref{ssec:numrsr}. For the remainder of this section, we will simply assume that $q\in H$ for some suitable normed vector space $H$ of functions from $\domPS$ to $\vsQOI$.

In practice, we cannot compute $\PDE_y^{-1}$ exactly but instead we have to rely on discretizations $\PDE_{y,\ns^{(2)}}^{-1}$, corresponding to a numerical solver with $\ns^{(2)}$ mesh points and coefficients determined by $y$. This yields approximations $\fq_{\ns^{(2)}}\in\vsPS$ of $\fq$, defined by
\begin{equation*}
\fq_{\ns^{(2)}}(y):=\qoi(\PDE_{y,\ns^{(2)}}^{-1}f_y).
\end{equation*}
Furthermore, we cannot compute approximations to the solution for all values of $y$, but only for $\ns^{(1)}$ samples and then need to rely on quadrature rules $\qr_{\ns^{(1)}}\colon \vsPS\to\R$ based on these samples. A straightforward approximation of $E[\fq]$ is then
\begin{equation}
\label{eq:UQtrivial}
E[\fq]\approx \qr_{\ns^{(1)}}\fq_{\ns^{(2)}}
\end{equation}
 for large $\ns^{(1)}$ and $\ns^{(2)}$, corresponding to many samples of a fine discretization of the PDE.

To obtain approximations that achieve the same error with less work, we observe that 
\begin{align*}
\bbl\colon \mathcal{L}(\vsPS,\R)\times \vsPS&\to\R\\
(\lambda,h)&\mapsto \lambda h
\end{align*}
is a continuous bilinear form and our goal is to approximate 
\begin{equation*}
\goal:=E[\fq]=\bbl(E,\fq)\in\R
\end{equation*}
using the approximations $\qr_{\ns}$ of $E$ and $\fq_{\ns}$ of $\fq$. This is exactly the setting of \Cref{sec:sparse}, with Assumption 1 on the continuity of the bilinear form corresponding to the definition of operator norms. To satisfy Assumptions 2 to 4 on error and work, we assume  
\begin{equation*}
\|E-\qr_{\ns}\|_{\vsPS\to\R}\lesssim_{\ns} \ns^{-\reg},
\end{equation*}  
\begin{equation*}
\|\fq-\fq_{\ns}\|_{H}\lesssim_{\ns} \ns^{-\kappa},
\end{equation*} and that an approximation of the solution of the PDE with a fixed parameter $y\in\domPS$ and $\ns$ mesh points requires the work $\ns^{\gamma}$ for some $\gamma>0$. Furthermore, we associate with $\qr_{\ns^{(1)}}$ the work $\ns^{(1)}$, with $\fq_{\ns^{(2)}}$ the work $(\ns^{(2)})^{\gamma}$, and with $\qr_{\ns^{(1)}}\fq_{\ns^{(2)}}$ the work $\ns^{(1)}(\ns^{(2)})^{\gamma}$ required for $\ns^{(1)}$ calls of the PDE solver with $\ns^{(2)}$ mesh points.
 \Cref{thm:central} now shows that for small enough $\epsilon$, we may choose $L$ such that Smolyak's algorithm
\begin{equation}
\label{eq:ML}
\smol:=\sum_{|\bm{l}|_1\leq L}\Delta_{l_1}^{(1)}\Delta_{l_2}^{(2)}=\sum_{l=1}^{L-1} \qr_{l}\Delta^{(2)}_{L-l}
\end{equation}
achieves an error of size $\epsilon>0$ with work 
\begin{equation}
\label{eq:MLres}
\epsilon^{-\rho}|\log\epsilon|^{(1+\rho)(n_0-1)},
\end{equation}
where 
\begin{equation*}
\rho:=\max\{\gamma/\kappa,1/\beta\},n_0:=|\argmax\{\gamma/\kappa,1/\beta\}|.
\end{equation*}
Using the straightforward approximation from \Cref{eq:UQtrivial} instead would require the work $\epsilon^{-\gamma/\kappa-1/\beta}$.

\Cref{eq:ML} is the celebrated \emph{multilevel formula} \cite{giles2008multilevel} and \Cref{eq:MLres} agrees with the work analysis in \cite{TeckentrupJantschWebsterEtAl2015}. Strictly speaking, when the so called weak convergence $E[q_N]\to E[q]$ occurs at a faster rate than the strong convergence $q_N\to q$, and when additionally $\gamma/\kappa>1/\beta$, slightly improved rates can be proven by a more elaborate analysis. We show in \Cref{sub:OUU} how such random sampling and corresponding probabilistic results can be obtained in the framework of multilinear approximation problems. First, we expand on the case of deterministic quadrature rules on domains in $\R^d$. Here, assuming separable probability densities, we can interpret the multidimensional integral operator as a tensor product of lower-dimensional integral operators and this multilinear structure allows for further sparsification of the approximation. More specifically, we consider the case where the integration domain is a cartesian product, $\domPS=\Pi_{j=1}^n\domPS^{(j)}$ with $\domPS^{(j)}\subset\R^{d_j}$, and where
\begin{equation*}
\vsPS=H^{\bm{\reg}}_{\pa}(\domPS)=\bigotimes_{j=1}^\dt H^{\reg_j}(\domPS^{(j)}),
\end{equation*} which is a tensor product Sobolev space as in \Cref{sec:kernel}. Furthermore, we assume that the distribution $\pi$ of $y$ is separable, $\pi=\prod_{j=1}^n \pi^{(j)}$. Since the operators $\bigotimes_{j=1}^\dt \int_{\domPS^{(j)}}d\pi^{(j)}$ and $\int_{\domPS}d\pi$ agree on elementary tensors by Fubini's theorem, they are equal, and we may consider the multilinear approximation problem

\begin{equation*}
E[\fq]=\left(\bigotimes_{j=1}^\dt \int_{\domPS^{(j)}}d\pi^{(j)}\right) \fq=:\bl(\int_{\domPS^{(1)}}d\pi^{(1)},\dots,\int_{\domPS^{(n)}}d\pi^{(n)},\fq).
\end{equation*}

This time we form Smolyak's algorithm based on the $(n+1)$-linear map $\bl$ and on quadrature rules $\qr^{(j)}_{\ns}\colon H^{\reg_j}(\domPS^{(j)})\to\R$.
We maintain the assumption that 
 \begin{equation}
 \label{eq:Qconvergence}
 \|\fq-\fq_{\ns}\|_{H^{\bm{\reg}}_{\pa}(\domPS)}\lesssim_{\ns} \ns^{-\kappa}.
 \end{equation}
and further assume that
\begin{equation}
\label{eq:QRconvergence}
\|\int_{\domPS^{(j)}}d\pi^{(j)}-\qr^{(j)}_{\ns}\|_{H^{\reg_j}(\domPS^{(j)})\to\R}\lesssim_{\ns} \ns^{-\beta_j/d_j}
\end{equation}
for $j\in\{1,\dots,\dt\}$.
 For example, if $\pi^{(j)}$ has a bounded density with respect to the Lebesgue measure, we can use for $\qr^{(j)}_{N}$ the integral over the kernel-based best-approximation associated with point evaluations in $Y^{(j)}_{\ns}\subset\domPS^{(j)}$ with  $h_{Y^{(j)}_{\ns},\domPS^{(j)}}\lesssim_{\ns} \ns^{-1/d_j}$. In any case, we assume that the resulting quadrature points and weights are calculated beforehand. 
 Smolyak's algorithm in this setting yields the Multi-index Stochastic Collocation method, which was introduced in  \cite{Haji-AliNobileTamelliniEtAl2015,Haji-AliNobileTamelliniEtAl2015a}. By presenting this method in the general framework of Smolyak's algorithm, we obtain a succinct proof of its convergence. Indeed, the convergence rate in \Cref{thm:main} below agrees with that in \cite[Theorem 1]{Haji-AliNobileTamelliniEtAl2015}.
\begin{thm}
\label{thm:main}
Let $\fq_N$, $\qr_N$, and $\bl$ be as above. In particular, assume that the estimates in \Cref{eq:Qconvergence,eq:QRconvergence} hold and that each call of the PDE solver to obtain a sample $\fq_{N}(y)$, $y\in\domPS$, requires the computational work $\tau\lesssim_{N}\ns^{\gamma}$. For small enough $\epsilon>0$, we can choose $L$ such that Smolyak's algorithm $\smoll_L:=\smoll_L(E[\fq])$ with work parameters $(1,\dots,1,\gamma)$ and convergence parameters $(\beta_1/d_1,\dots,\beta_{\dt}/d_{\dt},\kappa)$achieves
\begin{equation*}
|E[\fq]-\smoll_L|\leq \epsilon
\end{equation*}
with the computational work bounded by
\begin{equation}
\label{eq:workthm}
\work(\smoll_L)\lesssim_{\epsilon}\epsilon^{-\rho}|\log\epsilon|^{(1+\rho)(\dt_0-1)},
\end{equation}
where 
\begin{align*}
\rho:=\max \{d_1/\beta_1,\dots,d_\dt/\beta_{\dt},\gamma/\kappa\}\\
\intertext{and}
 \dt_0=|\argmax \{d_1/\beta_1,\dots,d_\dt/\beta_\dt,\gamma/\kappa\}|.
 \end{align*}
\end{thm}
\begin{proof}
We check the assumptions of \Cref{sec:sparse}. For this purpose, we view $\bl$ as a multilinear map 
\begin{equation*}
\bl\colon\left(\prod_{j=1}^\dt \mathcal{L}(H^{\reg_j}(\domPS_j), \R)\right)\times H^{\bm{\reg}}_{\pa}(\domPS)\to \R.
\end{equation*}
 Assumption 1 on the continuity of $\bl$ follows from \Cref{pro:cross} together with the definition of the operator norm: For arbitrary elements $\lambda^{(j)}\in \mathcal{L}(H^{\reg_j}(\domPS_j), \R)$ and $h\in H^{\bm{\reg}}_{\pa}(\domPS)$ we have
\begin{equation*}
\begin{split}
|\bl(\lambda^{(1)},\dots,\lambda^{(\dt)},h)|& \leq \|\bigotimes_{j=1}^\dt \lambda^{(j)}\|_{H^{\bm{\reg}}_{\pa}(\domPS)\to\R} \|h\|_{H^{\bm{\reg}}_{\pa}(\domPS)}\\
&\leq \prod_{j=1}^\dt \|\lambda^{(j)}\|_{H^{\reg_j}(\domPS_j)\to \R} \|h\|_{H^{\bm{\reg}}_{\pa}(\domPS)}.
\end{split}
\end{equation*}

Assumption 2 follows from \Cref{eq:Qconvergence,eq:QRconvergence}. Finally, we assign as work to $\qr^{(j)}_{\ns}$ the number of required point evaluations $\ns$, to $\fq_{\ns}$ the computational work $\ns^{\gamma}$, and to the evaluation of 
\begin{equation*}
\bl(\qr^{(1)}_{\ns^{(1)}},\dots,\qr^{(\dt)}_{\ns^{(\dt)}},\fq_{\ns^{(\dt+1)}})
\end{equation*}
the product 
 \begin{align*}
(\prod_{j=1}^\dt \ns^{(j)})(\ns^{(\dt+1)})^{\gamma},
\end{align*}
which is the computational work required by $\prod_{j=1}^\dt \ns^{(j)}$ calls of the PDE solver with $\ns^{(\dt+1)}$ mesh points. Therefore, all Assumptions of \Cref{sec:sparse} are satisfied and the claim follows from \Cref{thm:central}.
\end{proof}

\subsection{Response surface approximation}
\label{sub:RSR}

Our general formulation of Smolyak's algorithm allows us to extend the multilevel and multi-index methods of the previous subsection to the approximation of the full response surface $\fq\colon\domPS\to\R$ without much effort, provided we have an interpolation method which converges to the identity with an algebraic rate. We give below a result using kernel-based approximations (see \Cref{sec:kernel}).

As before, we assume that $\fq\in H^{\bm{\reg}}_{\pa}(\domPS)$, where $\domPS=\Pi_{j=1}^{\dt}\domPS^{(j)}$ and $\domPS^{(j)}\subset\R^{d_j}$ are Lipschitz domains and $\bm{\reg}=(\reg_1,\dots,\reg_{\dt})>(d_1/2,\dots,d_{\dt}/2)$. To apply Smolyak's algorithm we observe that
\begin{equation*}
\fq=\Id \fq=\left(\bigotimes_{j=1}^\dt \Id^{(j)}\right) \fq=:\bl(\Id^{(1)},\dots,\Id^{(\dt)},\fq),
\end{equation*}
where $\Id^{(j)}$, $j\in\{1,\dots,\dt\}$, is the identity on $H^{\reg_j}(\domPS_j)$, which we approximate by the best-approximations $S^{(j)}_{N}$ from \Cref{sec:kernel} based on evaluations in $Y^{(j)}_{N}\subset \domPS^{(j)}\subset\R^{d_j}$ such that $h_{Y^{(j)}_{N},\domPS^{(j)}}\lesssim_{N} N^{-1/d_j}$.
\begin{thm}
\label{thm:RSR}
Suppose that we have convergence $\fq_{N}\to \fq$ as specified in \Cref{eq:Qconvergence}, and that each call of the PDE solver to obtain a sample $q_N(y)$, $y\in\domPS$, requires the computational work $\tau\lesssim_{N}\ns^{\gamma}$. Let $0\leq \bm{\regs}< \bm{\reg}$. 
 For small enough $\epsilon>0$ we can choose $L$ such that Smolyak's algorithm $\smoll_L:=\smoll_L(\fq)$ with work parameters $(1,\dots,1,\gamma)$ and convergence parameters $((\reg_1-\regs_1)/d_1,\dots,(\reg_{\dt}-\regs_{\dt})/d_{\dt},\kappa)$ satisfies
\begin{equation*}
\|\fq-\smoll_{L}\|_{H^{\bm{\regs}}_{\pa}(\dom))}\leq \epsilon
\end{equation*}
and such that the computational work spent on calls of the PDE solver is bounded by
\begin{equation}
\label{eq:RSRwork}
W(\smoll_L)\lesssim_{\epsilon}\epsilon^{-\rho}|\log\epsilon|^{(1+\rho)(\dt_0-1)},
\end{equation}
where 
\begin{align*}
\rho:=\max \{d_1/(\beta_1-\alpha_1),\dots,d_\dt/(\beta_\dt-\alpha_\dt),\gamma/\kappa\}\\
\intertext{and}
 \dt_0=|\argmax \{d_1/(\beta_1-\alpha_1),\dots,d_\dt/(\beta_\dt-\alpha_\dt),\gamma/\kappa\}|.
 \end{align*}
\end{thm}
\begin{proof}
As before, Assumption 1 of \Cref{sec:sparse} holds since the multilinear map
\begin{equation*}
\begin{split}
\bl\colon\left(\prod_{j=1}^\dt \mathcal{L}(H^{\reg_j}(\domPS_j), H^{\regs_j}(\domPS_j))\right)&\times H^{\bm{\reg}}_{\pa}(\domPS)\to H^{\bm{\regs}}_{\pa}(\domPS)\\
\left((A^{(j)})_{j=1}^n,h\right)&\mapsto \Big(\bigotimes_{j=1}^n A^{(j)}\Big)h
\end{split}
\end{equation*}
is continuous by \Cref{pro:cross}.

Assumption 2 of \Cref{sec:sparse} holds for $\fq_{\ns}$ by assumption and for $S^{(j)}_{\ns}$, $j\in\{1,\dots,\dt\}$, by \Cref{pro:sobolev}, which states that
\begin{equation*}
\|\Id^{(j)}-S^{(j)}_{\ns}\|_{H^{\reg_j}(\domPS^{(j)})\to H^{\regs_j}(\domPS^{(j)})}\lesssim_{N} N^{-(\reg_j-\regs_j)/d_j}.
\end{equation*}
 Finally, we associate as work to $S^{(j)}_{\ns}$ the number of required point evaluations $\ns$, to $\fq_{\ns}$ the work $\ns^{\gamma}$, and to the evaluation of 
\begin{equation*}
\bl(S^{(1)}_{\ns^{(1)}},\dots,S^{(\dt)}_{\ns^{(\dt)}},q_{\ns^{(\dt+1)}})
\end{equation*}
the product 
 \begin{align*}
(\prod_{j=1}^\dt \ns^{(j)})(\ns^{(\dt+1)})^{\gamma},
\end{align*}
which is the computational work required for $\prod_{j=1}^\dt \ns^{(j)}$ samples of $\fq_{\ns^{(\dt+1)}}$. Therefore all Assumptions of \Cref{sec:sparse} are satisfied and the claim follows from \Cref{thm:central}.
\end{proof}

\begin{rem}\textbf{(Total work)}
 In practice, additionally to calls of the PDE solver, we need to determine the elements $s^{(j)}_{i,N}\in H^{\reg_j}(\domPS_j)$ in
 \begin{equation}
 S^{(j)}_{N}f=\sum_{i=1}^N s^{(j)}_{i,N} f(y^{(j)}_{i,N}),
 \end{equation}
 using \Cref{pro:pro}. Denote the computational work required for this computation by $\tau_{S}$ and denote by $\tau_{P}$ the computational work required for all calls of the PDE solver, for which we use the estimate 
\begin{equation}
\label{eq:taup}
\tau_{P}\lesssim_{L}\exp(Lg)L^{\dt_0-1}
\end{equation}
from \Cref{eq:work}, where $g=\max\bm{g}$ and $\dt_0=|\argmax \bm{g}|$ with 
\begin{equation*}
\bm{g}=(1/(1+(\beta_1-\regs_1)/d_1),\dots,1/(1+(\reg_{\dt}-\regs_\dt)/d_\dt),\gamma/(\gamma+\kappa)).
\end{equation*}

If we can solve the linear systems of \Cref{pro:pro} of size $N$ in time $N^{\lambda}$, then (recall the definition $N^{(j)}_{l}:=\exp(l/(1+(\beta_j-\regs_j)/d_j))$ from \Cref{sec:sparse})
\begin{equation}
\label{eq:taus}
\tau_{S}=\sum_{j=1}^{\dt}\sum_{l=1}^L \left(\ns^{(j)}_{l}\right)^{\lambda}\lesssim_{L}\sum_{l=1}^L \exp(l\tilde{g})\lesssim_{L}\exp(L\tilde{g}),
\end{equation}
where $\tilde{g}:=\lambda \max_{j=1}^\dt \{1/(1+(\beta_{j}-\regs_j)/d_j))\}$. Comparison of \Cref{eq:taup,eq:taus} shows that $\tau_{S}$ is negligible if 
\begin{equation*}
\tilde{g}\leq \gamma/(\gamma+\kappa).
\end{equation*}

However, our numerical experiments in Section\nobreakspace \textup {\ref {sec:numerics}} show that even in cases where $\tilde{g}>\gamma/(\gamma+\kappa)$, the work $\tau_{S}$ for the solution of the interpolation equations of a one-dimensional problem may be negligible compared to the cost $\tau_{P}$ of obtaining samples using PDE solvers in practical regimes of computation.  Intuitively, this may be explained by the fact that inverting the kernel matrices corresponding to one or two dimensional interpolation problems comes with no overhead, especially compared to the calls of the PDE solver, which requires meshing, preconditioning, etc.
\end{rem}

\subsection{Optimization under uncertainty}
\label{sub:OUU}
We now consider the case where the parameter $y=(z,m)$ in \Cref{eq:PDE} can be split into a deterministic component $z$ and a random component $m$. Taking expectation over the random component and optimizing over the deterministic one then gives rise to a problem of optimization under uncertainty \cite{Sahinidis2004,Shapiro2008,ShapiroDentchevaothers2014,AlexanderianPetraStadlerEtAl2016}.

We assume that $m$ is a random element over a probability space $(\Omega,\mathcal{A},P)$ with values in a possibly infinite-dimensional Banach space. 

Our goal is to solve the minimization problem
\begin{equation*}
\min_{z\in\domPS}\goal(z)+\psi(z),
\end{equation*}
where 
\begin{equation*}
\goal(z):=E[\qoi(u_{(z,m)})]
\end{equation*} is the expected value with respect to $P$, and $\psi$ depends only on the deterministic parameter $z$, acting as a penalty term for large $z$  for example. 
Difficulties in this minimization problem arise for similar reasons as before: 
\begin{itemize}
\item Given $z$ and $m$, we can only compute approximations $u_{(z,m),N}:=P^{-1}_{(z,m),N}f_{(z,m)}$ to the solution of the PDE.
\item Since $m$ is a random element, we need to rely on sampling strategies to approximate the expected value. In this section we show how Monte Carlo sampling can be included in the framework of general multilinear problems; however, deterministic sampling strategies as in \Cref{ssec:exp} may be applied alternatively.
\item As a consequence we can only compute approximations of $\goal(z)$, and we can only do so for few values of $z$. 
\end{itemize}

To address these issues, we propose a method for the computation of surrogate models that converge to $Q$ with high probability. These surrogate models are given in terms of their coefficients with respect to a basis of kernel functions. As such they can be evaluated with relatively low computational effort and minimized by standard techniques.

 As in previous sections, we assume that $\domPS=\prod_{j=1}^n \domPS^{(j)}$ for Lipschitz domains $\domPS^{(j)}\subset\R^{d_j}$, $\goal\in H^{\bm{\reg}}_{\pa}(\domPS)$ for some partition $D$ of $\{1,\dots,d\}$, and $\bm{\reg}=(\reg_1,\dots,\reg_{\dt})>(d_1/2,\dots,d_{\dt}/2)$ to assure that pointwise evaluations are possible. Furthermore, we assume that for $P$-almost all $\omega\in\Omega$ we have 
\begin{equation*}
\tilde{\goal}(\omega):=(z\mapsto \qoi(u_{(z,m(\omega))})\in H^{\bm{\reg}}_{\pa}(\domPS)
\end{equation*}
and 
$\tilde{\goal}$ is an element of the Bochner space $L^2(\Omega;H^{\bm{\reg}}_{\pa}(\domPS))$.
To obtain a multilinear approximation problem, we consider the problem of approximating the constant $H^{\bm{\reg}}_{\pa}(\domPS)$-valued random variable
\begin{equation*}
\goal=\bl(\Id^{(1)},\dots,\Id^{(n)},E,\tilde{\goal}):=\left(\Id^{(1)}\otimes\dots \otimes\Id^{(n)}\right)E[\tilde{\goal}]\in L^2(\Omega^{\N};H^{\bm{\reg}}_{\pa}(\domPS)),
\end{equation*}
where 
\begin{itemize}
\item $(\Omega^{\N},\mathcal{A}^{\N},P^{\N})$ is the $\N$-fold product probability space with product measure, which represents a sequence of independent and identically distributed draws of $m$ 
\item we regard expectation as an operator
\begin{equation*}
E\colon L^2(\Omega;H^{\bm{\reg}}_{\pa}(\domPS))\to L^2(\Omega^{\N};H^{\bm{\reg}}_{\pa}(\domPS)),
\end{equation*}
which maps elements of the Bochner space $L^2(\Omega;H^{\bm{\reg}}_{\pa}(\domPS))$ to their expected value, regarded as deterministic element of $L^2(\Omega^{\N};H^{\bm{\reg}}_{\pa}(\domPS))$
\item $\Id^{(j)}$ is the identity on $H^{\reg_j}(\domPS_j)$ for $j\in\{1,\dots,\dt\}$.

\end{itemize}

\vspace{1em}
To apply Smolyak's algorithm, we use the following approximations
\begin{itemize}
\item  To approximate $\tilde{\goal}$, we use the random variable
\begin{equation*}
\tilde{\goal}_N(\omega):=(z\mapsto \qoi(u_{(z,m(\omega)),N}))
\end{equation*}
\item  To approximate the expectation operator $E$, we use the empirical mean operators
\begin{align*}
\emp_{N}\colon L^2(\Omega;H^{\bm{\reg}}_{\pa}(\domPS))&\to L^2(\Omega^{\N};H^{\bm{\reg}}_{\pa}(\domPS)),\\
(\emp_{N}X)(\bm{\omega})&:=\frac{1}{N}\sum_{i=1}^N X(\omega_i)\quad\forall \bm{\omega}\in \Omega^{\N}
\end{align*}
\item  To approximate the identities $\Id^{(j)}$, we use the best-approximations $S^{(j)}_{N}$ from \Cref{sec:kernel} based on sets $Z^{(j)}_{N}\subset \domPS^{(j)}\subset\R^{d_j}$ with fill-distances
\begin{equation*}
	h_{Z^{(j)}_{N},\Gamma^{(j)}}\lesssim_{N} N^{-1/d_j}.
\end{equation*}
\end{itemize}
Smolyak's algorithm applied to this setting yields random elements in $H^{\bm{\reg}}_{\pa}(\domPS)$ that converge to the deterministic function $\goal\in H^{\bm{\reg}}_{\pa}(\domPS)$ in the probabilistic mean squared error (MSE) as $L\to\infty$.

Before we prove this convergence, we describe Smolyak's algorithm from a computational perspective.
If we denote the members of $Z^{(j)}_{N}$ by $(z^{(j)}_{i,N})_{i=1}^N$, then by \Cref{pro:pro} each $S^{(j)}_{N}$ can be written as
\begin{equation}
\label{eq:ugly}
S^{(j)}_{N}f=\sum_{i=1}^N s^{(j)}_{i,N} f(z^{(j)}_{i,N}),
\end{equation}
with $s^{(j)}_{i,N}\in H^{\reg_j}(\domPS^{(j)})$.
The value $\mathcal{M}(S^{(1)}_{N^{(1)}},\dots,S^{(n)}_{N^{(n)}},\op_{N^{(n+1)}},\tilde{\goal}_{N^{(n+2)}})$, which is formally defined as an element of $L^2(\Omega^{\N};H^{\bm{\reg}}_{\pa}(\domPS))$, is therefore given by
\begin{equation}
\label{eq:simple}
\sum_{\bm{i}\in \ind}s_{\bm{i}} \frac{1}{N^{(n+1)}}\sum_{k=1}^{N^{(n+1)}}\qoi\left(u_{(z_{\bm{i}},m(\omega_k)),N^{(\dt+2)}}\right)
\end{equation}
where
\begin{equation*}
\ind:=\{\bm{i} \in\N^{\dt}: 1\leq i_j\leq N^{(j)}\},
\end{equation*}
\begin{equation*}
s_{\bm{i}}:=\bigotimes_{j=1}^{\dt} s^{(j)}_{i_j,N^{(j)}},
\end{equation*}
and
\begin{equation*}
z_{\bm{i}}:=(z^{(1)}_{i_1,N^{(1)}},\dots,z^{(\dt)}_{i_{\dt},N^{(\dt)}}).
\end{equation*}
This means that we draw independent samples $(m(\omega_k))_{k=1}^{N^{(n+1)}}$ of $m$ and then form a kernel interpolant based on the averaged values $(\frac{1}{N^{(n+1)}}\sum_{k=1}^{N^{(n+1)}}\qoi\left(u_{(z_{\bm{i}},m(\omega_k)),N^{(\dt+2)}}\right))_{\bm{i}\in\ind}$. From the combination rule in \Cref{eq:combinationrule}, we see that Smolyak's algorithm is a linear combination of approximations as in \Cref{eq:simple}. 
\begin{thm}
\label{thm:OUU}
 Let $0\leq \bm{\regs}\leq  \bm{\reg}$. 
Assume that 
 \begin{equation*}
\sup_{m} \|\qoi(u_{(\cdot,m)})-\qoi(u_{(\cdot,m),N})\|_{H^{\bm{\reg}}_{\pa}(\domPS)}\lesssim_{\ns} \ns^{-\kappa}
 \end{equation*}
  and assume that each call of the PDE solver to obtain $u_{(z,m),N}$ given $z$ and $m$ requires the computational work $\tau\lesssim_{N}\ns^{\gamma}$. 
 For small enough $\epsilon>0$ we can choose $L$ such that the Smolyak algorithm $\smoll_L:=\smol$ with work parameters $(1,\dots,1,1,\gamma)$ and convergence parameters $((\reg_1-\regs_1)/d_1,\dots,(\reg_{\dt}-\regs_{\dt})/d_{\dt},1/2,\kappa)$ satisfies
\begin{equation}
\label{eq:boundOUU}
\E[\|\goal-\smoll_{L}\|_{H^{\bm{\regs}}_{\pa}(\dom)}^2]\leq \epsilon^2
\end{equation}
and such that the computational work required for all calls of the PDE solver is bounded by
\begin{equation}
\work(\smoll_L)\lesssim_{\epsilon}\epsilon^{-\rho}|\log\epsilon|^{(1+\rho)(\dt_0-1)},
\end{equation}
where 
\begin{align*}
\rho:=\max \{d_1/(\beta_1-\alpha_1),\dots,d_\dt/(\beta_\dt-\alpha_\dt),2,\gamma/\kappa\}\\
\intertext{and}
 \dt_0=|\argmax \{d_1/(\beta_1-\alpha_1),\dots,d_\dt/(\beta_\dt-\alpha_\dt),2,\gamma/\kappa\}|.
 \end{align*}
  By Chebyshev's inequality, for example, \Cref{eq:boundOUU} implies
  \begin{equation*}
  P(\|\goal-\smoll_{L}\|_{H^{\bm{\regs}}_{\pa}(\domPS)}\geq \delta)\leq \frac{\epsilon^2}{\delta^2}\quad\forall \delta>0.
  \end{equation*}
\end{thm}
\begin{proof}
We check the assumptions of \Cref{sec:sparse} for the multilinear map 
\begin{equation*}
\begin{split}
\bl\colon\left(\prod_{j=1}^\dt \mathcal{L}(H^{\reg_j}(\domPS_j), H^{\regs_j}(\domPS_j))\right)&\times \mathcal{L}\left(L^2(\Omega;H^{\bm{\reg}}_{\pa}(\domPS)),L^2(\Omega^{\N};H^{\bm{\reg}}_{\pa}(\domPS))\right)\times L^2(\Omega,H^{\bm{\reg}}_{\pa}(\domPS))\to L^2(\Omega^{\N};H^{\bm{\regs}}_{\pa}(\domPS))\\
&\left((A^{(j)})_{j=1}^n,B,C\right)\mapsto \Big(\bigotimes_{j=1}^n A^{(j)}\Big)B[C]
\end{split}
\end{equation*}
 Assumption 1 on the continuity of $\bl$ follows from \Cref{pro:cross} together with the definition of operator norms:
\begin{equation*}
\begin{split}
\|\Big(\bigotimes_{j=1}^{\dt} A^{(j)}\Big)B[C]\|_{ L^2(\Omega^{\N};H^{\bm{\regs}}_{\pa}(\domPS))}& \leq \|\bigotimes_{j=1}^\dt A^{(j)}\|_{H^{\bm{\reg}}_{\pa}(\domPS)\to H^{\bm{\regs}}_{\pa}(\domPS)} \|B[C]\|_{L^2(\Omega^{\N};H^{\bm{\reg}}_{\pa}(\domPS))}\\
&\leq \prod_{j=1}^\dt \|A^{(j)}\|_{H^{\reg_j}(\domPS_j)\to H^{\regs_j}(\domPS_j)} \|B\|_{L^2(\Omega;H^{\bm{\reg}}_{\pa}(\domPS))\to L^2(\Omega^{\N};H^{\bm{\reg}}_{\pa}(\domPS)) }\|C\|_{L^2(\Omega;H^{\bm{\reg}}_{\pa}(\domPS))}.
\end{split}
\end{equation*}

Next, we check the convergence rates of the input approximations $S^{(j)}_{\ns}\to \Id^{(j)}$, $\op_{\ns}\to E$, and $\tilde{\goal}_{\ns}\to \tilde{\goal}$.
\begin{itemize}
\item By \Cref{pro:sobolev}, we have
\begin{equation*}
\|\Id^{(j)}-S^{(j)}_{\ns}\|_{H^{\reg_j}(\domPS^{(j)})\to H^{\regs_j}(\domPS^{(j)})}\lesssim_{N} N^{-(\reg_j-\regs_j)/d_j}.
\end{equation*}

\item By standard Monte Carlo theory, we have 
\begin{equation*}
\|E-\op_{\ns}\|_{L^2(\Omega;H^{\bm{\reg}}_{\pa}(\domPS))\to L^2(\Omega^{\N};H^{\bm{\reg}}_{\pa}(\domPS))}\lesssim_{N}N^{-1/2}.
\end{equation*}

\item We have
\begin{equation*}
\begin{split}
\|\tilde{\goal}-\tilde{\goal}_{\ns}\|_{L^2(\Omega;H^{\bm{\reg}}_{\pa}(\domPS))}&\leq \sup_{m} \|Q(u_{(\cdot,m)})-Q(u_{(\cdot,m),N})\|_{H^{\bm{\reg}}_{\pa}(\domPS)}\\
&\lesssim_{N} N^{-\kappa}
\end{split}
\end{equation*}
by assumption.
\end{itemize}

We associate as work with the operators $S^{(j)}_{\ns}$, $j\in\{1,\dots,\dt\}$, and $\op_{\ns}$ the number of required evaluations $\ns$, with $\tilde{v}_{\ns}$ the computational work $\ns^{\gamma}$ required per sample, and, to satisfy Assumption 4 of \Cref{sec:sparse}, with 
\begin{equation*}
\bl(S^{(1)}_{\ns^{(1)}},\dots,S^{(\dt)}_{\ns^{(\dt)}},\op_{\ns^{(\dt+1)}},\tilde{\goal}_{\ns^{(\dt+2)}})
\end{equation*}
the product 
 \begin{align*}
(\prod_{j=1}^{\dt+1}\ns^{(j)})(\ns^{(\dt+2)})^{\gamma},
\end{align*}
which is the computational work required for the calls of the PDE solver (cf. \Cref{eq:simple}).
\end{proof}

\begin{rem}\textbf{(Total work)}
\label{rem:linf}
By the same arguments as in the remark after \Cref{thm:RSR}, the work for the computation of the elements $s^{(j)}_{i,N}$ in \Cref{eq:ugly} is negligible when 
\begin{equation*}
\lambda \min_{j=1}^\dt \{1/(1+(\beta_{j}-\regs_j)/d_j))\}\leq \max\{2/3,\gamma/(\gamma+\kappa)\},
\end{equation*}
where $N^{\lambda}$ is the computational work required to solve linear systems of size $N$ in \Cref{pro:pro}.
\end{rem}
\begin{rem}\textbf{(Convergence in $L^\infty$)}
	 Analogously to the remark after \Cref{thm:spkernel}, the result also holds if we measure the error in $L^{\infty}(\domPS)$, if we change the convergence parameters to
\begin{equation*}
\left((\beta_1-d_1/2)/d_1,\dots,(\beta_\dt-d_\dt/2)/d_{\dt},1/2,\kappa\right)
\end{equation*}
and adapt $\rho$ and $n_0$ accordingly.

\end{rem}
\begin{rem}
	The recent work \cite{dereich2015general} describes a similar algorithm employing a multilevel approach to directly find minima without reconstruction of the complete response surface using the Robbins-Monro algorithm \cite{robbins1951stochastic}.
	\end{rem}
\section{Numerical Experiments}
\label{sec:numerics}
To support our theoretical analysis, we performed numerical experiments on two linear elliptic
PDEs. The computation times presented below were achieved using MATLAB’s Parallel Com-
puting Toolbox for parallel computations on two Intel Xeon X5650 (2.66Ghz) processors with a
combined number of 12 cores.
\subsection{Response surface approximation}
\label{ssec:numrsr}
We consider diffusion through a material with a number $\dt$ of uncertain bumps in the diffusion coefficient, described by the partial differential equation
\begin{equation}
\label{eq:PDEnum}
\begin{aligned}
-\nabla \cdot (a_y \nabla u_{y})&=1&\text{ in }(0,1)^2\\
u_{y}&=0&\text{ on } \partial (0,1)^2,
\end{aligned}
\end{equation}
where the parameter $y=(c_1,\dots,c_\dt)\in\domPS^{(1)}\times\dots\times\domPS^{(\dt)}\subset ([0,1]^2)^n$ describes the centers of the bumps. These affect the diffusion coefficient additively through
\begin{equation}
a_y(x)=2+\sum_{j=1}^\dt\phi(x-c_j)
\end{equation}
where $\phi\in C^2_c(\R^2,\Rnn)$ is defined by $\phi(x):=\phi_0(|x|/R)$ with $\phi_0(r):=\int_{0}^{(1-r)_{+}} s^2(1-s)^2ds$ and $R=0.25$ for $n=1$, $R=0.125$ otherwise. We are interested in the dependence of the spatial average
\begin{equation}
\label{eq:spatialaverage}
\fq(y):=\qoi(u_y):=\int_{(0,1)^2} u_y(x)\;dx
\end{equation}
on the locations of the bumps.

For our numerical experiments, we partitioned the variables according to the bumps they describe. This means that we used kernel-based approximations
 $S^{(j)}_{N}\colon H^{\regnum}(\domPS^{(j)})\to L^2(\domPS^{(j)})$, $j\in\{1,\dots,\dt\}$ based on evaluations in sets $Y^{(j)}_{N}\subset \domPS^{(j)}\subset[0,1]^2$ with $|Y^{(j)}_{N}|=N$ and $h_{Y^{(j)}_{N},\domPS^{(j)}}\lesssim_{N}N^{-1/2}$. The domains $\domPS^{(j)}\subset[0,1]^2$ were chosen such that the supports of $\phi(x-c_j)$, $c_j\in\domPS^{(j)}$ did not overlap, see \Cref{fig:domainRSR}.
 
 \begin{figure}
 	\centering
 	\begin{subfigure}[b]{0.3\textwidth}
 \begin{tikzpicture}
\draw (0,0) rectangle (4,4);
\fill[black!10!white] (1,1) rectangle (3,3);
\node at (2,2) {$\domPS^{(1)}$};
\end{tikzpicture}
 \caption{n=1}
 \end{subfigure}
 	\begin{subfigure}[b]{0.3\textwidth}
  \begin{tikzpicture}
\draw (0,0) rectangle (4,4);
\fill[black!10!white] (0.5,0.5) rectangle (1.5,3.5);
\node at (1,2) {$\domPS^{(1)}$};
\fill[black!10!white] (2.5,0.5) rectangle (3.5,3.5);
\node at (3,2) {$\domPS^{(2)}$};
\end{tikzpicture}
  \caption{n=2}
  \end{subfigure}
  	\begin{subfigure}[b]{0.3\textwidth}
    \begin{tikzpicture}
\draw (0,0) rectangle (4,4);
\fill[black!10!white] (0.5,0.5) rectangle (1.5,1.5);
\node at (1,1) {$\domPS^{(1)}$};
\fill[black!10!white] (2.5,0.5) rectangle (3.5,1.5);
\node at (3,1) {$\domPS^{(2)}$};
\fill[black!10!white] (0.5,2.5) rectangle (1.5,3.5);
\node at (1,3) {$\domPS^{(3)}$};
\fill[black!10!white] (2.5,2.5) rectangle (3.5,3.5);
\node at (3,3) {$\domPS^{(4)}$};
\end{tikzpicture}
    \caption{n=4} 
    \end{subfigure}
       
      \caption{Domain of the numerical example in \Cref{ssec:numrsr} and domains $\domPS^{(j)}$ of the centers $c_j$, $j\in\{1,\dots,\dt\}$.}
      \label{fig:domainRSR}
 \end{figure} 
 
 Furthermore, we approximated $u_y$ using the finite element method with continuous piecewise linear elements on a quasi-uniform mesh with maximal element size $h_{\text{max}}$ and $M\approx h_{\text{max}}^{-2}$ mesh points. Consequently, instead of sampling $q$, we relied on the approximations $q_{h_{\max}}(y):=Q(u_{y,h_{\max}})$.

  Differentiation of \Cref{eq:PDEnum} with respect to $y$ and subsequent calculations similar to but simpler than those in \cite{Kuo2012} show that
   \begin{equation}
   \begin{split}
    \|\partial_{\bm{\alpha}}\fq-\partial_{\bm{\alpha}}\fq_{h_{\text{max}}}\|_{L^2(\domPS)}\lesssim_{h_{\text{max}}} h_{\text{max}}^{2}
    \end{split}
    \end{equation}
     for all $\bm{\regs}$ such that $\alpha_{2j-1}+\alpha_{2j}\leq 2$, $j\in\{1,\dots,n\}$ (observe that under this condition on $\bm{\regs}$ the derivative $\partial_{\bm{\regs}} a_{y}$ with respect to $y$ exists). Therefore, we have
    \begin{equation}
        \|\fq-\fq_{h_{\text{max}}}\|_{H^{\bm{\regnum}}_{\pa}(\domPS)}\lesssim_{h_{\text{max}}}h_{\text{max}}^{2}\approx M^{-1},
        \end{equation}
        where $H^{\bm{\regnum}}_{D}(\domPS)=\bigotimes_{j=1}^n H^2(\domPS^{(j)})$.

   Finally, we verified experimentally that samples of $\fq_{h_{\text{max}}}$ required approximately the runtime $h_{\text{max}}^{-3}=M^{3/2}$. Consequentially, we applied Smolyak's algorithm from \Cref{sub:RSR} with work parameters $(1,\dots,1,3/2)$ and convergence parameters $(\regnum/2,\dots,\regnum/2,1)=(1,\dots,1,1)$.  \Cref{thm:RSR} shows that the runtime required for an error of size $\|\fq-\smoll_{L}(\fq)\|_{L^2(\domPS)}\leq \epsilon$ is asymptotically bounded by $\epsilon^{-3/2}$.
Observe in particular that the exponent is independent of the number of bumps $\dt$, whereas straightforward approximation would require the work $\epsilon^{-\dt-3/2}$.
  
 \Cref{fig:convergenceRSR} below shows convergence plots for $\dt\in\{1,2,4\}$, which exhibit the expected rate of convergence. The runtimes include the solution of the linear systems required for kernel-based approximation. Even though the associated work is not asymptotically negligible for the given work and convergence parameters, this did not affect the results in the ranges of $L$ we considered. For each $\dt\in\{1,2,4\}$, a reference solution was computed with Smolyak's algorithm using $L_{\max}\gg 1$ and the $L^2(\Gamma)$-distance to the results of Smolyak's algorithm with $n+1\leq L<L_{\max}$ was approximated using $10^4$ random evaluations.
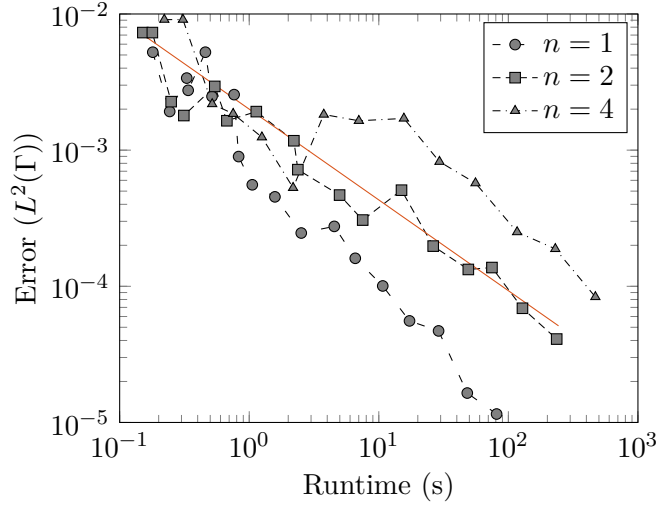
\begin{figure}
	\centering
%
%
\definecolor{mycolor1}{rgb}{0.00000,0.44700,0.74100}%
\definecolor{mycolor2}{rgb}{0.85000,0.32500,0.09800}%
\pgfplotscreateplotcyclelist{my black white}{%
	loosely dashed, every mark/.append style={solid, fill=gray}, mark=*\\%
	densely dashed, every mark/.append style={solid, fill=gray}, mark=square*\\%
	dashdotted, every mark/.append style={solid, fill=gray}, mark=triangle*\\%
	dashed, every mark/.append style={solid, fill=gray},mark=diamond*\\%
	loosely dashed, every mark/.append style={solid, fill=gray},mark=*\\%
	densely dashed, every mark/.append style={solid, fill=gray},mark=square*\\%
	dashdotted, every mark/.append style={solid, fill=gray},mark=otimes*\\%
	dasdotdotted, every mark/.append style={solid},mark=star\\%
	densely dashdotted,every mark/.append style={solid, fill=gray},mark=diamond*\\%
}
\begin{tikzpicture}

\begin{axis}[%
width=4.521in,
height=3.566in,
at={(0.758in,0.481in)},
scale only axis,
xmode=log,
xmin=0.1,
xmax=1000,
xminorticks=true,
xlabel={Runtime (s)},
ymode=log,
ymin=1e-05,
ymax=0.01,
yminorticks=true,
ylabel={Error ($L^2(\Gamma)$)},
axis background/.style={fill=white},
scale=0.6,
cycle list name=my black white
]
\addplot+[]
  table[row sep=crcr]{%
0.180149931824674	0.00523977493805551\\
0.243290831862168	0.00192352704599864\\
0.329372286556565	0.00338232346490313\\
0.337826642963775	0.00275037981662938\\
0.458810708184196	0.00523977493805551\\
0.516235428534473	0.00247889970302947\\
0.764487160634312	0.00255718575241824\\
0.828232438212983	0.000897020339534786\\
1.05560966295751	0.000556618113209411\\
1.57867725238421	0.000453915960935669\\
2.52048955350743	0.000246203346459327\\
4.53301980125603	0.000275422997131164\\
6.58195859898357	0.000160579291911704\\
10.7272709774825	0.000100589396027051\\
17.2624022505942	5.57107738757943e-05\\
28.928974860204	4.70325063402668e-05\\
48.1711874363383	1.64448669141017e-05\\
81.1448789204662	1.15323536363487e-05\\
};
\addplot+[]
table[row sep=crcr]{%
0.151045881109314	0.00730333096544007\\
0.178815296834063	0.00730333096544008\\
0.249592141398168	0.00227065835749392\\
0.312811187066783	0.00179533225718381\\
0.542123959758465	0.00293829003765648\\
0.672276085820071	0.0016437334979963\\
1.13376376306245	0.00191886948807182\\
2.2102602518984	0.00117198491034611\\
2.37306416333746	0.00071911020540633\\
4.98209233195032	0.000467883122284005\\
7.51320661165933	0.00030670717393757\\
14.9714398890717	0.000508106694528483\\
26.2796801341487	0.000197415826145638\\
49.0725795834337	0.000132963055766953\\
74.7307652035597	0.000137331500597793\\
128.451008078641	6.89909567261404e-05\\
236.057771178813	4.09975450698723e-05\\
};

\addplot+[]
table[row sep=crcr]{%
0.221046658716087	0.00909428287070589\\
0.307839123167523	0.00909428287070589\\
0.517837991403211	0.00218240376993264\\
0.751647409992016	0.00185078622741262\\
1.25261672378099	0.00124405601098679\\
2.18013722834769	0.000529139574664358\\
3.76226559197818	0.00182704212620214\\
7.01455587489616	0.00164372350535247\\
15.6196024658604	0.00170943563290494\\
29.3935360089679	0.000822660062181611\\
55.7455062412567	0.000572613670607768\\
116.830164906529	0.000250352456070469\\
230.264069870663	0.000188300530007596\\
469.294394994984	8.355366732799e-05\\
};
\addplot [color=mycolor2,solid]
table[row sep=crcr]{%
	0.15  0.0070834390
	0.445237446318393	0.00343006386939841\\
	8.83907641113316	0.000467834311355433\\
	17.2329153759479	0.00029977183433457\\
	25.6267543407627	0.000230091685342141\\
	34.0205933055775	0.000190488490431948\\
	42.4144322703922	0.000164445082174835\\
	50.808271235207	0.000145794245779252\\
	59.2021102000218	0.00013166546310175\\
	67.5959491648365	0.000120526593449001\\
	75.9897881296513	0.000111479027961758\\
	84.3836270944661	0.000103957988992388\\
	92.7774660592809	9.75891775093615e-05\\
	101.171305024096	9.21138838913963e-05\\
	109.56514398891	8.73470790537403e-05\\
	117.958982953725	8.31526286294459e-05\\
	126.35282191854	7.9427952089126e-05\\
	134.746660883355	7.60941589208973e-05\\
	143.140499848169	7.3089497495666e-05\\
	151.534338812984	7.03648808741522e-05\\
	159.928177777799	6.78807558796765e-05\\
	168.322016742614	6.56048648276327e-05\\
	176.715855707429	6.35106148193281e-05\\
	185.109694672243	6.15758694358012e-05\\
	193.503533637058	5.97820397298371e-05\\
	201.897372601873	5.81133909457469e-05\\
	210.291211566688	5.65565071537168e-05\\
	218.685050531502	5.50998731151365e-05\\
	227.078889496317	5.37335443008143e-05\\
	235.472728461132	5.24488839811811e-05\\
	243.866567425947	5.12383519044914e-05\\
};
\addlegendentry{$\dt=1$}
\addlegendentry{$\dt=2$}
\addlegendentry{$\dt=4$}
\end{axis}
\end{tikzpicture}%
\caption{Convergence of Smolyak's algorithm for response surface reconstruction. Expected slope: $-2/3$ (red).}
\label{fig:convergenceRSR}
\end{figure} 
 
The vertical shifts in the convergence plots of Figure\nobreakspace \ref{fig:convergenceRSR} represent growth of constants with respect to $n$, which is hidden in the notation $\lesssim_{\epsilon}$ in Theorem\nobreakspace \ref{thm:central}.

\subsection{Optimization under uncertainty}

We consider an advection-diffusion problem with a deterministic source term $f$, user-controlled velocity $z\in B_1(0)\subset\R^2$, and random diffusion coefficient  $a_m:=1+\exp(-m)$, where $m\sim \mathcal{N}(0,k)$ is a centered Gaussian random field on $[0,1]^2$ with covariance $k(x,y)=\exp(-(10|x-y|)^2)$. We are interested in the spatial average $\qoi(u_{(z,m)}):=\int_{(0,1)^2}u_{(z,m)}$ of the solution of 
\begin{align*}
-\nabla \cdot (a_m \nabla u_{(z,m)})&=-z\cdot \nabla u_{(z,m)} +f&\text{in }(0,1)^2\\
\partial_{\mathbf{n}}\nabla u_{(z,m)}+u_{(z,m)}&=u_b&\text{ on }\partial(0,1)^2\\
\end{align*}
where
\begin{align*}
u_b:=\begin{cases}
1&\text{ on }\{0\}\times [0,1]\\
0&\text{ on }\{1\}\times [0,1]\\
\frac{1+\cos(\pi x_1)}{2}&\text{ on }[0,1]\times \{0\}\\
\exp(1-1/(1-x_1))&\text{ on }[0,1]\times \{1\}.
\end{cases}
\end{align*}
and $f(x):=20\exp(-|x-(0.5,0.5)|_2^2)$.

 Our goal is to minimize the expected value of the spatial average plus a quadratic penalty term for large velocities that accounts for expensive power consumption required for the generation of large velocities:
\begin{equation*}
\min_{z\in B_1(0)} \phi(z),
\end{equation*}
where 
\begin{equation*}
\phi(z):=v(z)+\psi(z):=E_m[\int_{(0,1)^2}(u_{(z,m)})(x)\,dx]+\frac{|z|^2}{10}.
\end{equation*}

For our numerical experiments, we used kernel-based approximations $S_N\colon H^{4}(B_1(0))\to L^{\infty}(B_1(0))$ based on evaluations in $Y_{N}\subset B_1(0)$ with $h_{Y_N,B_1(0)}\lesssim_{N} N^{-1/2}$, satisfying
 \begin{equation}
 \|S_N-\Id\|_{H^4(B_1(0))\to L^{\infty}(B_1(0))}\lesssim_{N}h_{Y_N,B_1(0)}^{3}\lesssim_{N}N^{-3/2}.
 \end{equation}
Furthermore, we used finite element approximations of $u_{(z,m)}$ with maximal element size $h_{\text{max}}$ and $M\approx h_{\text{max}}^{-2}$ mesh points. The finite element approximations converged at the rate $h_{\text{max}}^2\approx M^{-1}$ and required the computational work $M^{3/2}$. Consequentially, we applied Smolyak's algorithm from \Cref{thm:OUU} (using \Cref{rem:linf}) with work parameters $(1,1,3/2)$ and convergence parameters $(3/2,1/2,1)$.

 \Cref{fig:diffestim} below shows that, as predicted by the theory, the runtime required to achieve the bound 
 \begin{equation}
 \label{eq:msqOUU}
 E[\|\goal-\smoll_L(\goal)\|_{L^{\infty}(B_1(0))}^2]\lesssim_{\epsilon} \epsilon^2,
 \end{equation}
 is bounded up to constants by $\epsilon^{-2}$, which is an essential improvement on the work $\epsilon^{-2/3-3/2-2}$ that a straightforward approximation would require for the same bound.
 A reference solution was computed with Smolyak's algorithm using $L_{\max}:=15$ and the mean-squared $L^\infty(B_1(0))$-distance to the results of Smolyak's algorithm with $ 3\leq L<L_{\max}$ was approximated using $10^4$ evaluations in $B_1(0)$ and $20$ stochastic repetitions.
 
Using the surrogate model obtained by Smolyak's algorithm with $L=15$, we obtain the optimal velocity $z^*\approx (-0.451,-0.062)$ with $\goal(z^*)\approx 5.038$ and $\phi(z^*)\approx 5.059$.

\begin{figure}[b!]
\centering
%
%
\definecolor{mycolor1}{rgb}{0.00000,0.44700,0.74100}%
\definecolor{mycolor2}{rgb}{0.85000,0.32500,0.09800}%
\pgfplotscreateplotcyclelist{my black white}{%
	loosely dashed, every mark/.append style={solid, fill=gray}, mark=*\\%
	densely dashed, every mark/.append style={solid, fill=gray}, mark=square*\\%
	dashdotted, every mark/.append style={solid, fill=gray}, mark=triangle*\\%
	dashed, every mark/.append style={solid, fill=gray},mark=diamond*\\%
	loosely dashed, every mark/.append style={solid, fill=gray},mark=*\\%
	densely dashed, every mark/.append style={solid, fill=gray},mark=square*\\%
	dashdotted, every mark/.append style={solid, fill=gray},mark=otimes*\\%
	dasdotdotted, every mark/.append style={solid},mark=star\\%
	densely dashdotted,every mark/.append style={solid, fill=gray},mark=diamond*\\%
}
\begin{tikzpicture}

\begin{axis}[%
width=4.521in,
height=3.566in,
at={(0.758in,0.481in)},
scale only axis,
xmode=log,
xminorticks=true,
xlabel={Runtime (s)},
ymode=log,
yminorticks=true,
ylabel={Error (MSE)},
axis background/.style={fill=white},
cycle list name=my black white,
scale=0.8
]
\addplot+
  table[y expr=\thisrowno{1}^2,row sep=crcr]{
0.54773504704521	0.360650395489588\\
0.878236018283768	0.348865231427516\\
1.41751478415253	0.334007249499424\\
2.39596891728651	0.0895695713051015\\
4.64957073796423	0.0813898191897779\\
10.2551924499772	0.0880234357367589\\
20.4849313665785	0.0637635122046947\\
41.8899871194443	0.0429247039834979\\
91.2006664321015	0.0335913740512899\\
181.394944940051	0.0210503171158568\\
377.135811717533	0.018921512415869\\
798.893758897246	0.0149818898756054\\
};
\addplot [color=mycolor2,solid,forget plot]
  table[y expr=\thisrowno{1}^2,row sep=crcr]{
0.475028053716853	0.458266591997893\\
28.4969302692818	0.0591668961860585\\
56.5188324848468	0.0420127622317703\\
84.5407347004118	0.0343514298056074\\
112.562636915977	0.0297701277308319\\
140.584539131542	0.0266384558482751\\
168.606441347107	0.0243243221278562\\
196.628343562672	0.0225244796505104\\
224.650245778237	0.0210729034912554\\
252.672147993802	0.0198700582860689\\
280.694050209367	0.0188521646305296\\
308.715952424932	0.0179762165619937\\
336.737854640497	0.0172120205882368\\
364.759756856062	0.0165376717853647\\
392.781659071627	0.0159368405363821\\
420.803561287192	0.0153970714075035\\
448.825463502757	0.0149086762259097\\
476.847365718322	0.0144639906818341\\
504.869267933886	0.0140568612962114\\
532.891170149451	0.0136822828789224\\
560.913072365016	0.0133361369435054\\
588.934974580581	0.0130149994517343\\
616.956876796146	0.012715997168693\\
644.978779011711	0.0124366987401125\\
673.000681227276	0.01217503098967\\
701.022583442841	0.011929213813446\\
729.044485658406	0.0116977089776517\\
757.066387873971	0.0114791794417153\\
785.088290089536	0.0112724567416927\\
813.110192305101	0.0110765146119631\\
};
\end{axis}
\end{tikzpicture}%
\caption{Convergence of Smolyak's algorithm for optimization under uncertainty. Expected slope: $-1/2$ (red).}
\label{fig:diffestim}
\end{figure}
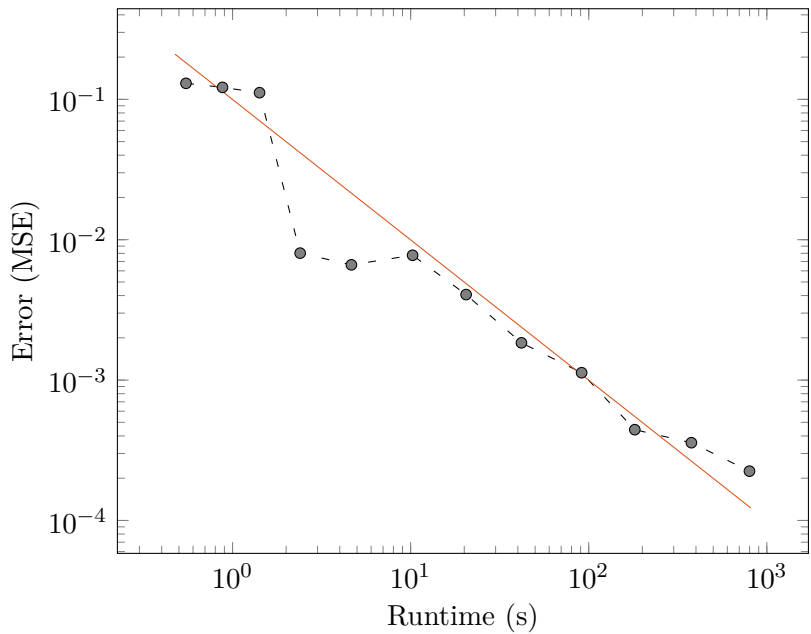
\section{Conclusion}

We have presented a framework for the sparse approximation of multilinear problems using Smolyak's algorithm and have shown complexity bounds that are, up to logarithmic factors, independent of the number of inputs. We have demonstrated how this framework can be used to obtain and analyze fast kernel-based algorithms for a number of problems in uncertainty quantification. In particular, for the problem of high-dimensional approximation, our framework generalizes results on sparse wavelet approximation from \cite{GriebelHarbrecht2013} to different approximation schemes, as we have illustrated for the case of kernel-based approximation. Furthermore, our results permit a general analysis of multilevel algorithms, extending in this respect the work of \cite{harbrecht2013multilevel}. Finally, we believe that our analysis may be helpful for the analysis of more general numerical approximation problems, where discretization parameters do not correspond to the number of interpolation nodes or basis functions. The generality of our arguments may also help in designing general purpose software that can be used to accelerate existing numerical implementations in a non-intrusive fashion.
\paragraph{Acknowledgement} S. Wolfers and R. Tempone are members of the KAUST Strategic Research Initiative, Center for Uncertainty Quantification in Computational Sciences and Engineering. R. Tempone received support from the KAUST CRG3 Award Ref: 2281. F. Nobile received support from the Center for ADvanced MOdeling Science (CADMOS).
\clearpage
\end{document}